\documentclass[12pt]{amsart}

\usepackage[margin=1in,marginparwidth=0.8in, marginparsep=0.1in]{geometry}
\pdfoutput=1

\usepackage{etoolbox}

\makeatletter
\let\old@tocline\@tocline
\let\section@tocline\@tocline
\newcommand{\subsection@dotsep}{4.5}
\newcommand{\subsubsection@dotsep}{4.5}
\patchcmd{\@tocline}
  {\hfil}
  {\nobreak
     \leaders\hbox{$\m@th
        \mkern \subsection@dotsep mu\hbox{.}\mkern \subsection@dotsep mu$}\hfill
     \nobreak}{}{}
\let\subsection@tocline\@tocline
\let\@tocline\old@tocline

\patchcmd{\@tocline}
  {\hfil}
  {\nobreak
     \leaders\hbox{$\m@th
        \mkern \subsubsection@dotsep mu\hbox{.}\mkern \subsubsection@dotsep mu$}\hfill
     \nobreak}{}{}
\let\subsubsection@tocline\@tocline
\let\@tocline\old@tocline

\let\old@l@subsection\l@subsection
\let\old@l@subsubsection\l@subsubsection

\def\@tocwriteb#1#2#3{%
  \begingroup
    \@xp\def\csname #2@tocline\endcsname##1##2##3##4##5##6{%
      \ifnum##1>\c@tocdepth
      \else \sbox\z@{##5\let\indentlabel\@tochangmeasure##6}\fi}%
    \csname l@#2\endcsname{#1{\csname#2name\endcsname}{\@secnumber}{}}%
  \endgroup
  \addcontentsline{toc}{#2}%
    {\protect#1{\csname#2name\endcsname}{\@secnumber}{#3}}}%

\newlength{\@tocsectionindent}
\newlength{\@tocsubsectionindent}
\newlength{\@tocsubsubsectionindent}
\newlength{\@tocsectionnumwidth}
\newlength{\@tocsubsectionnumwidth}
\newlength{\@tocsubsubsectionnumwidth}
\newcommand{\settocsectionnumwidth}[1]{\setlength{\@tocsectionnumwidth}{#1}}
\newcommand{\settocsubsectionnumwidth}[1]{\setlength{\@tocsubsectionnumwidth}{#1}}
\newcommand{\settocsubsubsectionnumwidth}[1]{\setlength{\@tocsubsubsectionnumwidth}{#1}}
\newcommand{\settocsectionindent}[1]{\setlength{\@tocsectionindent}{#1}}
\newcommand{\settocsubsectionindent}[1]{\setlength{\@tocsubsectionindent}{#1}}
\newcommand{\settocsubsubsectionindent}[1]{\setlength{\@tocsubsubsectionindent}{#1}}

\renewcommand{\l@section}{\section@tocline{1}{\@tocsectionvskip}{\@tocsectionindent}{}{\@tocsectionformat}}%
\renewcommand{\l@subsection}{\subsection@tocline{2}{\@tocsubsectionvskip}{\@tocsubsectionindent}{}{\@tocsubsectionformat}}%
\renewcommand{\l@subsubsection}{\subsubsection@tocline{3}{\@tocsubsubsectionvskip}{\@tocsubsubsectionindent}{}{\@tocsubsubsectionformat}}%
\newcommand{\@tocsectionformat}{}
\newcommand{\@tocsubsectionformat}{}
\newcommand{\@tocsubsubsectionformat}{}
\expandafter\def\csname toc@1format\endcsname{\@tocsectionformat}
\expandafter\def\csname toc@2format\endcsname{\@tocsubsectionformat}
\expandafter\def\csname toc@3format\endcsname{\@tocsubsubsectionformat}
\newcommand{\settocsectionformat}[1]{\renewcommand{\@tocsectionformat}{#1}}
\newcommand{\settocsubsectionformat}[1]{\renewcommand{\@tocsubsectionformat}{#1}}
\newcommand{\settocsubsubsectionformat}[1]{\renewcommand{\@tocsubsubsectionformat}{#1}}
\newlength{\@tocsectionvskip}
\newcommand{\settocsectionvskip}[1]{\setlength{\@tocsectionvskip}{#1}}
\newlength{\@tocsubsectionvskip}
\newcommand{\settocsubsectionvskip}[1]{\setlength{\@tocsubsectionvskip}{#1}}
\newlength{\@tocsubsubsectionvskip}
\newcommand{\settocsubsubsectionvskip}[1]{\setlength{\@tocsubsubsectionvskip}{#1}}

\patchcmd{\tocsection}{\indentlabel}{\makebox[\@tocsectionnumwidth][l]}{}{}
\patchcmd{\tocsubsection}{\indentlabel}{\makebox[\@tocsubsectionnumwidth][l]}{}{}
\patchcmd{\tocsubsubsection}{\indentlabel}{\makebox[\@tocsubsubsectionnumwidth][l]}{}{}

\newcommand{\@sectypepnumformat}{}
\renewcommand{\contentsline}[1]{%
  \expandafter\let\expandafter\@sectypepnumformat\csname @toc#1pnumformat\endcsname%
  \csname l@#1\endcsname}
\newcommand{\@tocsectionpnumformat}{}
\newcommand{\@tocsubsectionpnumformat}{}
\newcommand{\@tocsubsubsectionpnumformat}{}
\newcommand{\setsectionpnumformat}[1]{\renewcommand{\@tocsectionpnumformat}{#1}}
\newcommand{\setsubsectionpnumformat}[1]{\renewcommand{\@tocsubsectionpnumformat}{#1}}
\newcommand{\setsubsubsectionpnumformat}[1]{\renewcommand{\@tocsubsubsectionpnumformat}{#1}}
\renewcommand{\@tocpagenum}[1]{%
  \hfill {\mdseries\@sectypepnumformat #1}}

\let\oldappendix\appendix
\renewcommand{\appendix}{%
  \leavevmode\oldappendix%
  \addtocontents{toc}{%
    \protect\settowidth{\protect\@tocsectionnumwidth}{\protect\@tocsectionformat\sectionname\space}%
    \protect\addtolength{\protect\@tocsectionnumwidth}{2em}}%
}
\makeatother

\makeatletter
\settocsectionnumwidth{2em}
\settocsubsectionnumwidth{2.5em}
\settocsubsubsectionnumwidth{3em}
\settocsectionindent{1pc}%
\settocsubsectionindent{\dimexpr\@tocsectionindent+\@tocsectionnumwidth}%
\settocsubsubsectionindent{\dimexpr\@tocsubsectionindent+\@tocsubsectionnumwidth}%
\makeatother

\settocsectionvskip{5pt}
\settocsubsectionvskip{0pt}
\settocsubsubsectionvskip{0pt}
    
\settocsectionformat{\bfseries}
\settocsubsectionformat{\mdseries}
\settocsubsubsectionformat{\mdseries}
\setsectionpnumformat{\bfseries}
\setsubsectionpnumformat{\mdseries}
\setsubsubsectionpnumformat{\mdseries}

\let\oldtableofcontents\tableofcontents
\renewcommand{\tableofcontents}{%
  \vspace*{-\linespacing}% Default gap to top of CONTENTS is \linespacing.
  \oldtableofcontents}

\usepackage{soul}

\usepackage{amsfonts,amssymb,latexsym,amsmath,graphicx}
\input xy
\xyoption{all}

\usepackage[backref=page, bookmarks=true, bookmarksopen=true,%
bookmarksdepth=3,bookmarksopenlevel=2,%
colorlinks=true,%
linkcolor=blue,%
citecolor=blue,%
filecolor=blue,%
menucolor=blue,%
urlcolor=blue]{hyperref}

\usepackage{mathtools}
\usepackage{cleveref}

\usepackage[all, knot, poly]{xy}
\xyoption{knot}

\usepackage{times}
\usepackage{mathrsfs}

\usepackage{tikz-cd}

\newtheorem{lemma}{Lemma}[section]
\newtheorem{proposition}[lemma]{Proposition}
\newtheorem{corollary}[lemma]{Corollary}
\newtheorem{theorem}[lemma]{Theorem}

\theoremstyle{definition}
\newtheorem{definition}[lemma]{Definition}
\newtheorem{definition-proposition}[lemma]{Definition-Proposition}

\theoremstyle{remark} 
\newtheorem{remark}[lemma]{Remark}
\newtheorem{example}[lemma]{Example}

\newcommand{\C}{\mathbb{C}}

\renewcommand{\H}{\mathbb{H}}

\newcommand{\N}{\mathbb{N}}

\renewcommand{\P}{\mathbb{P}}

\newcommand{\R}{\mathbb{R}}
\renewcommand{\S}{\mathbb{S}}

\newcommand{\Z}{\mathbb{Z}}

\newcommand{\cC}{\mathcal{C}}
\newcommand{\cD}{\mathcal{D}}

\newcommand{\cF}{\mathcal{F}}
\newcommand{\cG}{\mathcal{G}}
\newcommand{\cH}{\mathcal{H}}
\newcommand{\cI}{\mathcal{I}}

\newcommand{\cK}{\mathcal{K}}

\newcommand{\cO}{\mathcal{O}}

\newcommand{\cS}{\mathcal{S}}
\newcommand{\cT}{\mathcal{T}}
\newcommand{\cU}{\mathcal{U}}
\newcommand{\cV}{\mathcal{V}}

\newcommand{\msh}{\mu sh}

\newcommand{\Pmsh}{\operatorname{\mathbb{P} \mu sh}}
\renewcommand{\c}{\mathfrak{c}}

\newcommand{\Coh}{\operatorname{Coh}}

\newcommand{\Map}{\operatorname{Map}}
\newcommand{\pre}{{pre}}
\newcommand{\Pic}{\operatorname{Pic}}
\newcommand{\Sq}{\operatorname{Sq}}
\newcommand{\supp}{\operatorname{supp}}

\DeclareMathOperator{\Hom}{Hom}

\DeclareMathOperator{\re}{re}
\DeclareMathOperator{\im}{im}

\makeatletter
\let\@wraptoccontribs\wraptoccontribs
\makeatother

\author{Laurent C\^ot\'e}
\address{%
Mathematical Institute\\
University of Bonn\\
Endenicher Allee 60\\
53115 Bonn\\
Germany
}
\email{lcote@math.uni-bonn.de}

\author{Christopher Kuo}
\address{%
Max Planck Institute for Mathematics\\
Vivatsgasse 7\\
53111 Bonn\\
Germany
}
\email{kuo@mpim-bonn.mpg.de}

\author{David Nadler}
\address{%
Department of Mathematics\\
University of California, Berkeley\\
970 Evans Hall\\
Berkeley, CA 94720-3840\\
USA
}
\email{nadler@math.berkeley.edu}

\author{Vivek Shende}
\address{%
Centre for Quantum Mathematics\\
Syddansk Universitet\\
Campusvej 55\\
5230 Odense\\
Denmark
}
\email{vivek.vijay.shende@gmail.com}
\contrib[\MakeLowercase{with an appendix by}]{Sanath Devalapurkar}



\title{Perverse microsheaves}

\begin{document}

\begin{abstract}
On a complex contact manifold, or complex symplectic manifold with weight-1 circle action, we construct a sheaf of
stable categories carrying a $t$-structure which is locally equivalent to a microlocalization of 
the perverse $t$-structure. 
\end{abstract}

\maketitle

\tableofcontents

\newpage

\section{Introduction}

For a complex manifold $M$, let us write $sh_{\C-c}(M)$ for a derived category of sheaves on $M$, whose objects 
are each locally constant on the strata of a locally finite stratification by complex subvarieties.  
Perverse sheaves are those $F$ with the following property: 
 
\begin{equation} \label{perverse}
\dim_\C \{\, x \in M \, | \,  H^i( \iota_x^*  F) \ne 0 \} \le -i  \qquad \qquad \dim_\C \{\, x \in M \, | \,  H^i( \iota_x^!  F) \ne 0 \} \le i 
\end{equation}
 Here we denote by $\iota_x: \{x\} \hookrightarrow M$ the inclusion of the point $p$ to $M$ and by $\iota_x^*$ and $\iota_x^!$ the induced restriction functors.  

Perverse sheaves have played a pivotal role in many results in algebraic geometry and geometric representation theory.  
They turn out to be natural both in terms of
considerations of Frobenius eigenvalues in positive characteristic \cite{bbd},
and in terms of analytic considerations in characteristic
zero, where they are the sheaves of solutions to  regular holonomic differential equations,
or more generally, D-modules \cite{kashiwara1979faisceaux, kashiwara-kawai, kashiwara-RH, mebkhout-RH}.
This latter equivalence highlights a key
feature, not immediately apparent from the definition: 
despite being a seemingly arbitrarily demarcated subcategory of a 
category of complexes, perverse sheaves form an {\em abelian} category.  

There are microlocal versions (living on $T^* M$ or $\P T^* M$) of the category of $D$-modules \cite{sato-kashiwara-kawai, kashiwara-microlocal-D}, perverse sheaves \cite{andronikov-microperverse, waschkies-microperverse}, and the equivalence between them \cite{andronikof-microlocal-RH, neto-microlocal-RH, waschkies-microlocal-RH}.  More generally still, 
Kashiwara has constructed a sheaf of categories on any complex contact manifold, locally
equivalent to the microlocalization of $D$-modules \cite{kashiwara-quantization-contact}, see also \cite{polesello-schapira}. 
A variant of this construction appropriate to conic complex symplectic geometry has allowed the methods of geometric representation
theory to be extended beyond cotangent bundles to more general symplectic resolutions and similar spaces \cite{kashiwara-rouquier, BPLW-I}. 

\vspace{2mm}
The purpose of the present article is to 
construct perverse $t$-structures on 
categories of complex-constructible microsheaves, globalizing the construction of Waschkies \cite{waschkies-microperverse}.   In the sequel \cite{CKNS2} we establish a Riemann-Hilbert equivalence with the canonical stack of $\mathcal{E}$-modules defined by Kashiwara \cite{kashiwara-quantization-contact}.

\vspace{2mm}

Our starting point is the globalization \cite{shende-microlocal, nadler-shende} of the microlocal sheaf theory of Kashiwara and Schapira \cite{kashiwara-schapira}.  We recall the relevant notions in Section \ref{section:sheaves-on-manifolds}.  In brief, the theory takes as input 
a {\em real} contact or exact symplectic manifold $V$, 
a choice of symmetric monoidal stable compactly generated coefficient category $\mathcal{C}$, and a  trivialization of a certain canonical obstruction $V \to B^2 Pic(\mathcal{C})$.  We refer to said trivialization as a {\em Maslov datum}; this is also what is required to define Floer-theoretic invariants in the same target spaces.  The output of the theory is a 
sheaf of stable categories $\mu sh_V$ on $V$ \cite[Thm. 1.1]{nadler-shende}.  For a locally closed subset $X \subset V$, we write $\mu sh_X \subset \mu sh_V|_X$ for the subsheaf of full subcategories on objects locally supported in $X$.  

A Legendrian or conic Lagrangian $L\subset V$ determines an obstruction $L \to BPic(\mathcal{C})$, a choice of trivialization for which (a `secondary Maslov datum') yields an equivalence $$\mu sh_L \cong loc_L$$ with the sheaf of categories of local systems along $L$ \cite[Thm. 1.2]{nadler-shende}. In particular, if $D \subset V$ is a smooth Legendrian disk containing a point $p$, a choice of secondary Maslov datum for $D$ determines an equivalence
\begin{equation}\label{intro:non-canonical-microstalk}  (\mu sh_D)_p \xrightarrow{\sim} \cC.
\end{equation}
The space of secondary Maslov data for the disk $D$ is a torsor for $BPic(\mathcal{C})$, which acts on maps \eqref{intro:non-canonical-microstalk} in the evident way.

In this paper, we will be interested in  contact and conic symplectic manifolds which come from complex geometry. It turns out that such manifolds admit a \emph{canonical} choice of Maslov datum, yielding a canonical notion of $\mu sh$. 
For simplicity, we state this in the symplectic setting. By an \emph{exact complex symplectic manifold}, we mean a complex manifold $W$ along with a holomorphic $1$-form $\lambda$ such that $d\lambda$ is holomorphic symplectic.  The underlying real manifold of $W$ carries the real exact symplectic structure $\operatorname{re}(\lambda)$, so that we can meaningfully discuss Maslov data and microlocal sheaves on $W$.%  

\begin{theorem}\label{theorem:intro-1-test}
    Let $W$ be an exact complex symplectic manifold and $\mathcal{C}$ the category of modules over a (discrete) commutative ring $R$.  Then there is a canonical Maslov datum for $W$, with respect to which secondary Maslov data for a conic complex Lagrangian $L\subset W$ are identified with $R$-spin structures on $L$.
\end{theorem}
By an $R$-spin structure on a real symplectic manifold $W$, we mean a null-homotopy of the composition $W \to BU \xrightarrow{w_2} B^2\mathbb{Z}^\times \to B^2R^\times$. When $R=\mathbb{Z}$, this is a spin structure in the ordinary sense; in general, if such structures exist, they form a torsor for $\mathrm{H}^1(X, R^\times)$.  The proof of \Cref{theorem:intro-1-test} is in \Cref{subsection:canonicalmicrostalkcomplex}, where we also deduce, from the aforementioned general properties of microlocal sheaves: 
\begin{corollary}[The canonical microsheaf category]\label{corollary:test-microsheaf-intro}  
Let $W, \mathcal{C}$ be as in \Cref{theorem:intro-1-test}. There is a canonical sheaf of stable $\mathcal{C}$-linear categories $\mu sh_{W}$ on $W$.  
For a complex conic Lagrangian $L \subset W$, an $R$-spin structure $\sigma$ on $L$ determines an equivalence 
$\mu sh_L \cong loc_L$. 

\end{corollary}

 In particular, specializing to the neighborhood of a point gives: 
\begin{corollary}[The microstalk functor]\label{corollary:microstalk-intro-test}
    Let $W, \mathcal{C}$ be as in \Cref{theorem:intro-1-test}. If $X \subseteq W$ is a closed subset which locally around $p \in X$ is conic complex Lagrangian,  there is a functor
\begin{equation}\label{equation:microstalk-test} \omega_p^{-1}: (\mu sh_X)_p \xrightarrow{\sim} \mathcal{C},
\end{equation}
which is well-defined up to non-canonical invertible natural transformation.  In particular, given $K \in \mu sh_X(W)$, the isomorphism class of $\omega_p^{-1}(K)$ is well defined. 
\end{corollary}

We refer to $\omega_p^{-1}(K)$ as the \emph{microstalk} of $K$ at $p$.  In \Cref{subsection:canonicalmicrostalkcomplex}, we will give a more explicit characterization of this microstalk.

\begin{definition}\label{definition:putative-t-intro-test}
    Write $\mu sh_{W, \mathbb{C}-c} \subset \mu sh_{W}$ for the sheaf of full subcategories on objects whose (micro)support is a complex analytic Lagrangian subset of $W$.  We define 
\begin{equation} \label{microperverse}
(\mu sh_{W, \C-c})^{\le 0} \subset \mu sh_{W, \C-c}  \qquad \qquad (\mu sh_{W, \C-c})^{\geq 0} \subset \mu sh_{W, \C-c}
\end{equation}    
%$$(\mu sh_{W, \C-c})^{\le 0} \subset \mu sh_{W, \C-c}$$ 
as the sheaves of full subcategories on those objects
all of whose microstalks, as elements of $R-mod$, have cohomology concentrated in degrees $\le 0$, resp.\ $\geq 0$.
\end{definition}
We do not know in general whether $((\mu sh_{W, \C-c})^{\le 0}, (\mu sh_{W, \C-c})^{\ge 0})$ is a $t$-structure on $\mu sh_{W, \C-c}$.  However, we will use these subcategories to construct $t$-structures on certain related categories.

Recall that for a complex contact manifold $V$, its {\em symplectization} is the $\C^*$ bundle $\pi: \widetilde V \to V$ 
whose (local) holomorphic sections give (local) complex contact forms; it carries canonically an exact complex symplectic form 
(as we review in more detail in  \Cref{subsection:complex-contact-symplectic}) and hence we have the canonically defined $\mu sh_{\widetilde V}$. 

\begin{theorem}[\ref{theorem:main-comparison}]\label{theorem:main}
    Then the pair 
    $\left( (\pi_*\mu sh_{\tilde{V}, \mathbb{C}-c})^{\leq 0}, (\pi_*\mu sh_{\tilde{V}, \mathbb{C}-c})^{\geq 0} \right)$ 
    determines a $t$-structure on $\pi_*\mu sh_{\widetilde{V}, \mathbb{C}-c}$ 
(i.e., the sections of the above sheaves over any open set $U$ give a t-structure on $\pi_*\mu sh_{\tilde{V}, \mathbb{C}-c}(U)$).
\end{theorem}

For an exact complex symplectic manifold $(W, \lambda)$, 
we may use the contactization $(W \times \C, \lambda + dz)$ to define a sheaf of categories $\pi_* \mu sh_{\pi^{-1}(W \times \{0\})}$, i.e.\ the sheaf of full subcategories of $\pi_*\mu sh_{\widetilde{W \times \mathbb{C}}}$ on objects whose support is contained in $\pi^{-1}(W \times \{0\})$.  We do not understand the relationship of this with $\mu sh_W$ in general. 
However, if the Liouville flow on $(W, \lambda)$ integrates to a weight-1 $\C^*$-action, then $\mathbb{C}^*$ naturally acts on $W \times \mathbb{C}$ by contactomorphism.    Let $\gamma_\C$ be the set-theoretic identity on $W \times \mathbb{C}$ (resp.\ on $W$) where the source carries the Euclidean topology but the target is endowed with the $\mathbb{C}^*$ invariant topology. 

\begin{theorem}\label{theorem:main symplectic intro}
Let $W$ be a complex exact symplectic manifold whose Liouville vector field integrates to a weight-1 $\C^*$ action. Then the pair $\left ( \left((\gamma_\mathbb{C})_* \mu sh_{W, \C-c}\right)^{\geq 0}, \left((\gamma_\mathbb{C})_* \mu sh_{W, \C-c}\right)^{\leq 0} \right)$
 determines a $t$-structure on $(\gamma_\mathbb{C})_* \mu sh_{W, \C-c}$. 
    Moreover, the Hom sheaf of two objects in the heart is a ($\frac{1}{2} \dim W$-shifted)
perverse sheaf. 
\end{theorem} 

Finally, we recall that there is a comparison theorem \cite{GPS3} between microsheaves and Fukaya categories; consequently, our results can be translated to give $t$-structures on certain Fukaya categories.  We spell this out in Appendix \ref{app: Fukaya}.

\vspace{2mm} {\bf Acknowledgements.}
We wish to thank 
Roman Bezrukavnikov, Sanath Devalapurkar, Peter Haine, Benjamin Gammage, Sam Gunningham, Justin Hilburn, Wenyuan Li, Michael McBreen, Semon Rezchikov, 
Marco Robalo, Pierre Schapira and Filip {\v{Z}}ivanovi{\'c} for enlightening conversations. We are particularly grateful to Sanath Devalapurkar for help with \Cref{g-od} and to Wenyuan Li for help with \Cref{prop: C-c-Pmsh-equal-all}.

LC was partially supported by Simons Foundation grant 385573 (Simons Collaboration on Homological Mirror Symmetry).
DN was partially supported by NSF grant DMS-2101466. 
CK and VS were partially supported by: VILLUM FONDEN grant VILLUM Investigator 37814, 
and VS was supported in addition by Novo Nordisk Foundation 
grant NNF20OC0066298, Danish National Research Foundation grant DNRF157, and the USA NSF 
grant CAREER DMS-1654545. CK was supported by Max Planck Institude for Mathematics in Bonn during the revision of the paper. 
 
 \section{Complex contact and symplectic manifolds}
\label{subsection:complex-contact-symplectic}
We review some standard properties of complex contact and symplectic manifolds. A classical reference in the contact setting is \cite{kobayashi}. %should add others...

If $X$ is a complex manifold, its holomorphic tangent bundle is denoted by $\cT_X$. The holomorphic cotangent bundle is denoted by $\Omega_X$, and its exterior powers are denoted by $\Omega^k_X:= \bigwedge^k \Omega_X$.

\begin{definition}
A \emph{complex (or holomorphic) symplectic manifold} is a complex manifold $W$ along with a closed holomorphic $2$-form $\omega \in H^0(W,\Omega^2_W)$. 
\end{definition}
A complex symplectic manifold $(W, \omega)$ determines a family, parameterized by  $\hbar \in \C^*$,
of real symplectic manifolds $(W, \re(\hbar \omega))$. 
By an exact complex symplectic manifold, we mean a pair $(W, \lambda)$  where $W$ is a complex manifold and
$\lambda$ is a holomorphic 1-form such that $d \lambda = \partial \lambda$ is symplectic. 

\begin{example}\label{example:holomorphic-cotangent}
Let $M$ be a complex manifold. Then the holomorphic cotangent bundle $\Omega_M$ carries a canonical holomorphic $1$-form $\lambda_{can, \C}= ydx$, where $(x_1,\dots,x_n, y_1,\dots,y_n)$ are canonical holomorphic coordinates. Then $(\Omega_M, d\lambda_{can, \C})$ is a complex symplectic manifold.

Meanwhile, the real cotangent bundle $T^*M$ carries the canonical $1$-form $\lambda_{can}$. There is a natural identification $\Omega_M \to T^*M$ defined in local holomorphic coordinates by $$(x,y) \mapsto (\re(x), \im(x), \re(y), -\im(y)).$$ One computes that this identification pulls back $\lambda_{can}$ to $\re(\lambda_{can, \C})$. 
\end{example}

If $(X, \omega)$ is a complex symplectic manifold, a half-dimensional complex submanifold $L \subset X$ is said to be \emph{complex Lagrangian} if $\omega|_L=0$. Clearly complex Lagrangian submanifolds are automatically (real) Lagrangian with respect to $\re(\hbar \omega)$, for all $\hbar \in \C^*$. 

\begin{definition}
A \emph{complex (or holomorphic) contact manifold} is a complex manifold $V$ of complex dimension $2n+1$ along with a holomorphic hyperplane field $\cH \hookrightarrow \cT_V$ which is maximally non-integrable. Concretely, this means that if $\alpha \in \Omega_V(U)$ is a holomorphic $1$-form for which $\cH = \ker \alpha$ in some local chart $U \subset V$, then $\alpha \wedge (d\alpha)^n \neq 0$. 
\end{definition}

Given a complex contact manifold $(V, \xi)$, there is a holomorphic line bundle $\cT_V/ \cH \to V$.  (Local) complex contact forms are nonvanishing (local) nonvanishing holomorphic sections of $(\cT_V/ \cH)^{\vee}$.  Correspondingly, a global complex contact form is a global holomorphic trivialization of this line bundle (which does not typically exist). 

The bundle $(\cT_V/ \cH)^{\vee}$ is naturally a holomorphic sub-bundle of 
$\Omega_V$ (indeed, for $v \in V$, we have $(\cT_V/ \cH)^{\vee}_v =  \{\alpha \in \Omega_{V, v} \mid \alpha(\xi)=0\} \subset \Omega_{V, v}$). 
We consider the $\mathbb{C}^*$-bundle $$\pi: \widetilde{V}:= (\cT_V/ \cH)^{\vee} \setminus 0_V \to V$$ which we call the \emph{complex symplectization of $V$} and let $\lambda_{\widetilde{V}}$ denote the pullback of $\lambda_{can, \C}$ under the inclusion $\widetilde{V} \hookrightarrow \Omega_V$.

We also consider the projectivized $S^1$-bundle $p:\widetilde{V}/\mathbb{R}_+ \to V$ and set $\xi_\hbar:= \ker( \re(\hbar \lambda_{\widetilde{V}}))$ for $\hbar \in \C^*$. The relation between these spaces is summarized by the following diagram:

\begin{equation}
\begin{tikzcd}
\widetilde{V} \ar[r, "q" '] \ar[rr, bend left=20, "\pi"] & \widetilde{V}/\R_+ \ar[r, "p" '] & V.
\end{tikzcd}
\end{equation}

\begin{lemma}\label{lemma:organizing-symplectizations}
With the above notation:
\begin{enumerate}
\item[(i)] $(\widetilde{V}, \hbar \lambda_{\widetilde{V}})$ is an exact complex symplectic manifold
\item[(ii)] $(\widetilde{V}/\mathbb{R}_+,\xi_\hbar)$ is a real contact manifold
\item[(iii)] $(\widetilde{V}, \re(\hbar \lambda_{\widetilde{V}}) )$ is canonically isomorphic to the real symplectization of $(\widetilde{V}/\mathbb{R}_+,\xi_\hbar)$. (This isomorphism intertwines the $1$-forms and the $\mathbb{R}_+$-bundle structure over $\widetilde{V}/\mathbb{R}_+$). 
\end{enumerate}
\end{lemma}

\begin{proof}
We compute $\hbar \lambda_{\widetilde{V}}$ locally by choosing a holomorphic contact $1$-form $\alpha$ defined on some open set $U \subset V$. Such a choice induces a holomorphic embedding 
\begin{align*}\label{align:local-embedding}
\iota_\alpha: \C^* \times U &\hookrightarrow \widetilde{V} \subset \Omega_V \\
(z, x) &\mapsto z \alpha_x
\end{align*}

To compute the pullback of $\hbar \lambda_{\widetilde{V}}$ under $\iota_\alpha$, choose $(z, x) \in \C^* \times U$ and $d \iota_\alpha (v) \in T \widetilde{V}_{z \alpha_x} \subset T \Omega_V$. Then we have $
\hbar \lambda_{\widetilde{V}}(d \iota_\alpha (v))=\hbar \lambda_{can, \C}(d \iota_\alpha (v))= z \hbar \alpha_x(d \pi \circ  d \iota_\alpha(v))$. Hence 
\begin{equation} \label{equation:pullback-lambda}
 \iota_{\alpha}^*  \hbar \lambda_{\widetilde{V}} = z \hbar \alpha.
 \end{equation}

Both (i) and (ii) can be checked immediately using \eqref{equation:pullback-lambda}. The proof of (iii) is an exercise in chasing definitions. 
\end{proof}

\begin{example}[\Cref{example:holomorphic-cotangent} continued]
The complex projectivization $V:=(\Omega_X- 0_X)/ \C^*$ is a complex contact manifold with respect to $\lambda_{can, \C}$ and we have $\widetilde{V}=(\Omega_X- 0_X)$. The associated real projectivization $\widetilde{V}/\R_+$ carries a circle of real contact forms $\re(\hbar \lambda_{can, \C})$. 
\end{example}

\begin{example}\label{example:complex-symplectic}
Given an exact complex symplectic manifold $(X, \lambda)$, then $(X \times \C, \lambda + dz)$ is a complex contact manifold.  The contact form $\lambda + dz$ defines a section of the $\C^*$-bundle $\widetilde{X \times \C}= X \times \C \times \C^* \to X \times \C$. Similarly, $\re(\hbar(\lambda + dz))$ defines a section of $\widetilde{X \times \C}/ \mathbb{R}^+= X \times \C \times S^1$. 
\end{example}

Observe that there is a fiber-preserving $\mathbb{C}^*$-action on $\widetilde{V}$: over some fiber $\widetilde{V}_v \subset \Omega_{V, v}$, it sends $\alpha \mapsto z \alpha$ for $z \in \mathbb{C}^*$ and $v \in V$. For concreteness, we write $z=e^{t+ i \theta}$ for $(t, \theta) \in  \mathbb{R}_+\times S^1$ and let $\partial_t, \partial_\theta$ denote the vector fields on $\widetilde{V}$ generated by the $\mathbb{R}_+$ and $S^1$ actions.

We have the following morphisms of bundles over $\widetilde{V}$:
\begin{equation}\label{equation:inclusions-bundles}
 \pi^* \xi \equiv \operatorname{ker}(\lambda_{\widetilde{V}}) / \langle \partial_t, \partial_\theta \rangle \twoheadleftarrow \operatorname{ker}(\lambda_{\widetilde{V}})   \hookrightarrow  \operatorname{ker}(\operatorname{re}(\hbar \lambda_{\widetilde{V}})) \hookrightarrow T \widetilde{V}   
\end{equation}

\begin{lemma}\label{lemma:splitting-symplectization}
There is a splitting of vector bundles over $\widetilde{V}$ (well-defined up to contractible choice) $$\operatorname{ker}(\operatorname{re}(\hbar \lambda_{\widetilde{V}})) = \partial_t \oplus \langle \partial_\theta, X \rangle \oplus \pi^* \xi,$$ where $X$ is a non-vanishing section of $\operatorname{ker}(\re(\hbar \lambda_{\widetilde{V}})) / \operatorname{ker}(\hbar \lambda_{\widetilde{V}})$ and 
\begin{itemize}
\item[(i)] $\langle \partial_\theta, X \rangle$ is a trivial $2$-dimensional real symplectic vector bundle with respect to $d(\re (\hbar \lambda_{\widetilde{V}}))$
\item[(ii)] $\pi^* \xi$ is a complex symplectic vector bundle with respect to $d(\lambda_{\widetilde{V}})$
\end{itemize}
\end{lemma}

\begin{proof}
The existence of this splitting is essentially a restatement of \eqref{equation:inclusions-bundles}. The other statements can be checked locally using \eqref{equation:pullback-lambda}. 

\end{proof}

Given a complex contact manifold $(V, \cH)$ of complex dimension $2n+1$, a complex submanifold $L \subset V$ of complex dimension $n$ which is everywhere tangent to $\cH$ is said to be a \emph{complex Legendrian}. 
\begin{lemma}
If $L \subset (V, \xi)$ is complex Legendrian, then its preimage under $\widetilde{V} \to V$ is denoted by $\widetilde{L}$ and is complex exact Lagrangian.
The quotient $\widetilde{L}/ \R_+ \subset \widetilde{V}/ \R_+$ is a real Legendrian with respect to any of the $\C^*$ of contact forms $\re(\hbar \lambda_{can, \C})$. 

\qed
\end{lemma}

\section{Grading and orientation data} \label{g-od}

Consider the following inclusions of groups ($n/2$ ones defined only when $n$ even): 
$$
\xymatrix{
& U(n/2, \H) \ar[rd] & \\
U(n/2) \ar[ru] \ar[rd] & & \sqrt{SU}(n)  \ar[r] & U(n) \\
&  O(n, \R) \ar[ru] & 
}
$$
Here $\sqrt{SU}(n)$ is defined as the kernel of $U(n) \xrightarrow{det^2} U(1)$. Recall that $U(n)$ is the maximal compact subgroup of both $Sp(2n, \mathbb{R})$ and $GL(n, \C)$, which retract to it; correspondingly, $BU(n)$ is canonically homotopy equivalent to the spaces classifying either complex or symplectic vector bundles.  Meanwhile $U(n/2, \mathbb{H})$ is also known as the `compact symplectic group' $Sp(n/2)= Sp(n, \mathbb{C}) \cap U(n)$,  the maximal compact subgroup of $Sp(n, \mathbb{C})$.

\begin{definition} \label{standard structures} 
Consider a topological space $X$ carrying a hermitian bundle classified by a map $X \to BU(n)$.  
We give names to the following sorts of structures: 

\begin{itemize}
\item A {\em grading} is a lift to $B(\sqrt{SU}(n))$
\item A {\em polarization} is a lift to $BO(n, \R)$
\item A {\em quaternionic structure} is a lift to $BU(n/2, \H)$
\item A {\em complex polarization} is a lift to $BU(n/2)$. 
\end{itemize} 
\end{definition}

When $X$ is a symplectic manifold or contact manifold, by a polarization (etc.) on $X$, we always mean a polarization (etc.) on the symplectic tangent bundle of $X$ or the contact distribution.  

Concretely, a polarization of a symplectic vector bundle can be viewed as Lagrangian plane field, as follows by inspecting the fiber sequence $O(n) \to U(n) \to U(n)/O(n)= LGr(n)$. Similarly, a complex polarization of a complex symplectic vector bundle can be viewed as a complex Lagrangian plane field.  We freely pass between both viewpoints in this paper.

\begin{lemma} \label{fiber versus canonical grading} A polarization induces a grading.  A quaternionic structure induces a grading.  A complex polarization 
induces both a quaternionic structure and a polarization, each of which induces the same grading. 
\qed
\end{lemma}

We will refer to gradings induced by any of the above structures as {\em canonical} gradings.

Let us recall the relationship between null-homotopies and lifts. 
Recall that a fiber sequence $P \to Q \to R$  is a 
Cartesian diagram $P \xrightarrow{\sim} Q \times_R \bullet$, where $\bullet$ is a point.  
Mapping spaces preserve limits, 
so $Map(X, P) \xrightarrow{\sim} Map(X, Q) \times_{Map(X, R)} Map(X, \bullet)$.  That is, 
given such a fiber sequence of pointed spaces and a map $X \to Q$, a null-homotopy of the composite 
map $X \to Q \to R$ is equivalent to a lift  $X \to P$.  

\begin{example} 
A grading is a lift from $X \to BU(n)$ to $X \to B \sqrt{SU}(n)$, so equivalently, a null-homotopy of the composition $X \to BU(n) \xrightarrow{B det^2} BU(1)$.  Recalling that homotopy class of $B det: BU(n) \to BU(1) = B^2 \Z$ in $[BU(n), B^2 \Z] = \mathrm{H}^2(BU(n), \Z)$ is the universal first Chern class, we see that the obstruction to the existence of a grading is twice the first Chern class. 
\end{example}

We consider the classical stabilized compact groups 
$$Sp = \lim_{n \to \infty} U(n, \H) \qquad \qquad U = \lim_{n \to \infty} U(n, \C) \qquad \qquad O = \lim_{n \to \infty} O(n, \R).$$
We write similarly $SU, \sqrt{SU}$.  Note the inclusion $U(n, \mathbb{C}) \to O(2n, \mathbb{R})$ lands in $SO(2n, \mathbb{R})$;  correspondingly we have inclusions $U \subset  SO \subset O$. 
We also have  the Lagrangian Grassmannians
$$LGr_\C = Sp / U \qquad \qquad LGr = U / O$$
The natural inclusion $LGr_\C \to LGr$ is the limit of $U(n, \H) / U(n, \C) \hookrightarrow U(2n, \C) / O(2n, \R)$.  

There are evident stable analogues of the notions of \Cref{standard structures}, and the stable analogue of
\Cref{fiber versus canonical grading} also holds.

\begin{lemma}
Let $(V, \xi)$ be a complex contact manifold and consider the real contact manifold $(\widetilde{V}/\mathbb{R}_+, \xi_{\hbar})$. Then the contact distribution $\xi_{\hbar} \to \widetilde{V}/\mathbb{R}_+$ carries a stable quaternionic structure. 
\end{lemma}
\begin{proof}
It is equivalent to prove that $q^* \xi_{\hbar}$ carries a stable quaternionic structure, where $q: \widetilde{V} \to \widetilde{V}/\mathbb{R}_+$ is the quotient map. This follows from \Cref{lemma:splitting-symplectization}. 

\end{proof}

One virtue of having passed to stabilizations is that $U/O$ is an infinite loop space. The map $det^2$ descends to an equivalence $det^2: \tau_{\le 1}(U/O) \xrightarrow{\sim} B \Z$.  Thus a grading on $X \to BU(n)$ is equivalently a null-homotopy of the composition $X \to BU(n) \to BU \to B(U/O) \to \tau_{\le 2}B(U/O)$.  
\begin{definition} \label{def: grading/orientation}
    Grading/orientation data for $X \to BU(n)$ is a null-homotopy of the composition 
    $X \to BU(n) \to \tau_{\le 3}B(U/O)$.
\end{definition}
Note that a stable polarization for $X \to BU(n)$ is a null-homotopy of the composition $X \to BU \to B(U/O)$, so canonically provides grading/orientation data.  We refer to such grading/orientation data as {\em polarization grading/orientation data}.

Truncations give the fiber sequence (of infinite loop spaces)
$$ B^3  (\Z/2\Z) \to \tau_{\le 3} B(U/O) \to  \tau_{\le 2}B(U/O) = B \left(\tau_{\leq 1}(U/O) \right)   \stackrel{Bdet^2}{=} B^2 \Z.$$
We refer to the space of grading/orientation data lifting a given grading as orientation data. 

\begin{remark}
    It is presumably implicit in the construction of orientations on moduli spaces for Floer theory that there is a universal way to choose orientation data on all (real) symplectic manifolds.  In fact there are two such ways; and their existence is somewhat subtle from our present viewpoint, as pointed out to us by Sanath Devalapurkar.  The main source of the subtlety is the fact that the fiber sequence $B(\Z/2 \Z) \to \tau_{\le 1}(U/O) \to \Z$ splits as a sequence of topological spaces, but not as a sequence of infinite loop spaces; in fact,  $B^3(\Z/2\Z) \to B^2  \tau_{\le 1}(U/O) \to B^2(\Z)$ does not split as a sequence of spaces.    
    We record his arguments in Appendix \ref{sanath}.  This fact is not logically required for the present article, because the existence of orientation data for complex symplectic manifolds is less subtle, as will be shown presently. 
\end{remark}

\begin{lemma}\label{lemma:truncate-complex}
The composition $Sp \to U \to U/O \to \tau_{\le 2}(U/O)$ factors canonically through $\tau_{\le 2}(Sp) = 0$.    
\end{lemma}

\begin{definition} \label{def: quarternionic-grading/orientation}
    Given a stable quaternionic bundle $X \to BSp$, the
    {\em canonical grading/orientation datum} of the induced complex bundle $X \to BSp \to BU$ is the null-homotopy induced from the null-homotopy of $Sp \to \tau_{\le 2}(U/O)$ above. Similarly, the {\em canonical grading} is the null-homotopy induced from the similar null-homotopy of $Sp \rightarrow \tau_{\leq 1}(U/O)$. 
\end{definition}
We note that the grading induced by the canonical grading/orientation datum agrees with the canonical grading of \Cref{fiber versus canonical grading}.

Because $Sp \subset \sqrt{SU}$,  the map $Sp/U \to U/O \xrightarrow{det^2} U(1) = \tau_{\le 1}(U/O)$ is canonically null-homotopic.  So the map $Sp/U \to U/O \to \tau_{\le 2}(U/O)$ lifts to $Sp/U \to B^2(\Z/2\Z)$.
\begin{lemma}
    This lift is the composition 
    $Sp/U \to BU \xrightarrow{c_1} B^2 \Z \to B^2(\Z/2\Z)$, where $c_1$ is the first Chern class.
\end{lemma}
\begin{proof}
Since $\tau_{\le 2} Sp = 0$, any map $Sp/U \to B^2(\Z/2\Z)$ must factor through $BU$.   There's a unique nontrivial map in $[BU, B^2(\Z/2\Z)]$, and it's the reduction mod 2 of the first Chern class.  Finally, it is standard that $\tau_{\le 2}(Sp/U \to U/O)$ is nontrivial: one has the map between fiber sequences
$$
\begin{tikzpicture}
% Nodes
\node at (0,2) {$U$};
\node at (4,2) {$Sp$};
\node at (8,2) {$Sp/U$};
\node at (0,0) {$O$};
\node at (4,0) {$U$};
\node at (8,0) {$U/O$};

% Horizontal Arrows
\draw [->, thick] (0.5,2) -- (3.6,2) node [midway, above] {$ $};
\draw [->, thick] (4.4,2) -- (7.4,2) node [midway, above] {$ $};
\draw [->, thick] (0.5,0) -- (3.6,0) node [midway, above] {$ $};
\draw [->, thick] (4.4,0) -- (7.4,0) node [midway, above] {$ $};

% Vertical Arrows
\draw [->, thick] (0,1.7) -- (0,0.3) node [midway, right] {$ $}; 
\draw [->, thick] (4,1.7) -- (4,0.3) node [midway, right] {$ $};
\draw [->, thick] (8,1.7) -- (8,0.3) node [midway, right] {$ $};
\end{tikzpicture},
$$  
which induces a map between the long exact sequence of homotopy groups, which shows that $\pi_2(Sp/U) = \pi_1(U) = \Z$, $\pi_2(U/O) = \pi_1(O) = \Z/2$, and the map $\pi_2(Sp/U) \rightarrow \pi_2(U/O)$ is given by the quotient map $\Z \rightarrow \Z/2$.
\end{proof}

\begin{proposition}\label{spin polarization} 
Fix a stable complex polarization $X \to BU$ of a given $X \to BSp$.  
The space of homotopies between the canonical and polarization orientation data is equivalent
to the space of null-homotopies of the composition $X \to BU  \xrightarrow{c_1} B^2 \Z \rightarrow B^2(\mathbb{Z}/2\mathbb{Z})$.
\end{proposition}
\begin{proof}
First we recall some general facts about null-homotopies and exact triangles.
Given two null-homotopies  $n_1, n_2$ of a given map $f: Q \to R$, we produce a pointed ``comparison map'' 
$[n_1, n_2] \in Q \times S^1 \to R$ by taking $S^1 = [-\pi, \pi]$, taking the map $f$ on $Q \times 0$ and applying the null-homotopy
$n_1$ along $[-\pi, 0]$ and the null-homotopy $n_2$ along $[0, \pi]$.  Note that $\Hom(Q \times S^1, R) = \Hom(Q, \Omega R)$. 
A homotopy between $n_1$ and $n_2$ is a null-homotopy of $[n_1, n_2] \in \Hom(Q, \Omega R)$. 

Suppose now given any exact triangle $P \to Q \to R \to BP$ in a stable category 
(for the application here, the stable category of spectra).  
The compositions $P \to Q \to R$ and $Q \to R \to BP$ give two null-homotopies
of the composite map $P \to BP$.   
The definition of exact triangles \cite[Def.\ 1.1.2.11]{lurie-ha} promises
that the comparison of these two null-homotopies is identified with the identity
of $\Hom(P, \Omega BP) = \Hom(P, P)$.  

Given maps $X \xrightarrow{p} P \xrightarrow{s}  S$, we may 
compose to learn that the corresponding comparison between null-homotopies of 
$$X \xrightarrow{p} P \to Q \to R \to BP \xrightarrow{Bs} BS$$
given by
$[s \circ n_1 \circ p, s \circ n_2 \circ p] \in \Hom(X, \Omega BS)$ is identified with $s \circ p \in \Hom(X, S)$.  Here, $X$ need only be a space, not a spectrum.

Now consider 
a stable quaternionic vector bundle with a stable complex polarization, i.e. we have a lift $X \to BU \to BSp$.  
Consider the composition
\begin{equation} \label{polarization orientation data} 
X \to  BU \to  BSp \to B(Sp/ U) \to B^2 U  \xrightarrow{Bc_1} B^3 \Z \rightarrow B^3(\Z/2\Z)
\end{equation}
Now, the canonical orientation data factors through $BSp \to 0$, hence is induced by the null-homotopy of the sequence $BSp \to B(Sp/U) \to B^2 U$.  The polarization orientation data comes from the prescribed null-homotopy of $X \to B(U/O)$; since we have a complex polarization this factors through a null-homotopy of $X \to B(Sp/U)$, hence is induced from the null-homotopy of $BU \to BSp \to B(Sp/U)$.  The result follows. 
\end{proof} 

We recall that the mod 2 reduction of the first Chern class of a complex bundle is the second Stiefel-Whitney class $w_2$ of the underlying real bundle.  Thus, null-homotopies of $X \to BU \xrightarrow{w_2} B^2(\Z/2\Z)$ as above are the same as spin structures on the underlying real bundle classified by $X \to BU \to BSO$.

Let $V$ be a symplectic (resp.\ contact) manifold. Given a Lagrangian (resp.\ Legendrian) submanifold, the Weinstein neighborhood theorem symplectomorphically (resp.\ contactomorphically) identifies a neighborhood of $L \subset V$ with a neighborhood of the zero section in $T^*L$ (resp.\ in $J^1L$). This identification is canonical up to contractible choice. We write $\phi_L$ for the fiber polarization of $T^*L$ (resp.\ $J^1L$). 

\begin{definition}  \label{secondary Z-Maslov}
Let $m$ be a grading (resp. grading/orientation datum) on $V$.  Then by a {\em secondary grading} (resp.  grading/orientation datum) for $L \subset V$, we mean a homotopy between $m|_L$ and $\phi_L$. 

Given grading/orientation data on $V$ and a secondary grading on $L$, then by {\em secondary orientation data}, we mean a lift of said secondary grading on $L$ to secondary grading/orientation data on $L$.  
\end{definition} 

Note that the obstruction to existence of a secondary grading is a class in $[L, B\Z] = H^1(L, \Z)$, 
and the space of secondary gradings is a torsor for $\Map(L, \Z)$, hence in particular, is discrete.  
Thus we simply ask whether secondary gradings are equal, rather than discuss homotopies
between them. 

\begin{lemma} \label{complex legendrian maslov} 
Let $V$ be a complex symplectic manifold. Equip the underlying real symplectic manifold with the canonical
grading constructed in \Cref{fiber versus canonical grading}.   Let $L\subset V$ be a smooth complex Lagrangian.  
Then there is a canonical choice for the secondary grading of $L \subset V$. 
\end{lemma} 

\begin{proof}
On $L$, the fiber polarization provides a complex polarization of 
the restriction of the contact distribution, which, 
by Lemma \ref{fiber versus canonical grading}, agrees with the canonical grading.  
\end{proof}

\begin{lemma} \label{lem: spin}
Let $L, V$ be as in \Cref{complex legendrian maslov}. Fix the canonical grading/orientation data on  $V$, and the canonical grading on $L$.  Then secondary orientation data for $L \subset V$ is
equivalent to a spin structure on $L$.  
\end{lemma}
\begin{proof}
This is a special case of Proposition \ref{spin polarization}.  
\end{proof}

\section{Microsheaves on real contact manifolds}\label{section:sheaves-on-manifolds}

Here we review ideas from the microlocal theory of sheaves as formulated for cotangent bundles of manifolds in \cite{kashiwara-schapira} and
globalized to arbitrary contact manifolds in \cite{shende-microlocal, nadler-shende}.

\subsection{Sheaves on manifolds}

Let $M$ be  a real manifold.  
Fix a symmetric monoidal stable presentable compactly generated category, $\cC$.  The reader will not lose much of the point of the paper taking throughout $\cC$ to be the derived category of dg modules over some
commutative ring $R$.
We write $sh(M)$ for the (stable) category of sheaves on $M$ with values in $\cC$.  

In this subsection we review ideas from Kashiwara and Schapira \cite{kashiwara-schapira}.  Often these were originally formulated for bounded derived categories, viewed
as triangulated categories.  Modern foundations \cite{lurie-htt, lurie-ha} allow one to work directly in the stable setting, and in addition for the 
boundedness hypothesis to be removed for many purposes; we do so when appropriate without further comment.

\subsubsection{Microsupport}   
Given $F \in sh(M)$, we say that a smooth function $f: M \to \R$ has a cohomological $F$-critical point at $x \in M$ if $(j^!F)_x \neq 0$ for $j: \{f \geq 0\} \hookrightarrow M$ the inclusion. The {\em microsupport} of $F$ (also called the \emph{singular support}) is the closure of the locus of differentials of functions at their cohomological $F$-critical points. We denote it by $ss(F)$. 

The microsupport is easily seen to be conical and satisfy $ss(Cone(F \to G)) \subset ss(F) \cup ss(G)$. A deep result of \cite[Thm.\ 6.5.4]{kashiwara-schapira} is that the microsupport is \emph{coisotropic} (also called \emph{involutive}; see \cite[Def.\ 6.5.1]{kashiwara-schapira}). 
For a conic subset $K \subset T^*M$, we write $sh_K(M)$ for the full subcategory on objects microsupported in $K$.  
For any subset $\Lambda \subset S^*M$, we write $sh_\Lambda(M) := sh_{\R_+ \Lambda \cup 0_M}(M)$, with $0_M$ the zero section of the cotangent bundle. We write $T^\circ M:= T^*M- 0_M$; following the usual convention, we will not distinguish between subsets of $S^*M$ and conic subsets of $T^*M- 0_M$.

The assignment $U \mapsto sh(U)$ defines a sheaf of categories on $M$; we denote it $sh$.  Similarly, 
$U \mapsto sh_{K \cap T^*U}(U)$ defines a subsheaf of full subcategories, we denote it $sh_K$.  Similarly, $sh_\Lambda$. 

\subsubsection{Constructibility}

\begin{definition}
A \emph{stratification} of a topological space $X$ is a locally finite decomposition $X= \sqcup_\alpha X_\alpha$, where the $X_\alpha$ are pairwise disjoint locally closed subsets called \emph{strata}. This decomposition must satisfy the \emph{frontier condition}: the boundary $\overline{X}_\alpha \setminus X_\alpha$ is a union of other strata. 
\end{definition}

\begin{theorem}[Thm. 8.4.2 of \cite{kashiwara-schapira}] 
Let $M$ be a real analytic manifold, and $F$ a sheaf on $M$.  Then the following are equivalent:
\begin{itemize}
\item There is a subanalytic stratification $M = \coprod M_i$ such that $F|_{M_i}$ is locally constant
\item $ss(F)$ is subanalytic (singular) Lagrangian
\end{itemize} 
Sheaves satisfying these equivalent conditions are said to be $\R$-constructible.  We write 
$Sh_{\R, c}(M)$ for the category of $\R$-constructible sheaves. 
\end{theorem}

\begin{theorem}[Thm.\ 8.5.5 of \cite{kashiwara-schapira}] \label{theorem:complex-microsupport}
Let $M$ be a complex analytic manifold, and $F$ a sheaf on $M$.  Then the following are equivalent: 
\begin{itemize}
\item There is a complex analytic stratification $M = \coprod M_i$ such that $F|_{M_i}$ is locally constant
\item $ss(F)$ is contained in a closed $\C^*$-conic subanalytic isotropic subset
\item $ss(F)$ is a complex analytic (singular) Lagrangian
\end{itemize} 
Sheaves satisfying these equivalent conditions are said to be $\C$-constructible.  We write 
$Sh_{\C, c}(M)$ for the category of $\C$-constructible sheaves. 
\end{theorem}

\subsection{Microsheaves on cotangent bundles} 

We consider the presheaf of stable categories on $T^*M$: 
\begin{equation} \label{ks mush}
\mu sh_{T^*M}^{pre}(\Omega) := sh(M) / sh_{T^*M \setminus \Omega}(M)
\end{equation}

\begin{definition}
Let $\mu sh_{T^*M}$ be the sheaf of categories on $T^*M$ defined by sheafifying the presheaf $sh_{T^*M}^{pre}(\Omega)$ in \eqref{ks mush}. Similarly, let $\mu sh_{S^*M}$ be the presheaf of categories on $S^*M$ obtained by sheafifying $\mu sh^{pre}_{S^*M}$.
\end{definition}

In any sheaf of categories $\mathcal{X}$ on a topological space $T$, given $F, G \in \mathcal{X}(T)$, the assignment
$U \mapsto \Hom_{\mathcal{X}(U)}(F|_\Omega, G|_\Omega)$ is a sheaf on $T$; let us denote it as $\cH om_{\mathcal{X}}(F, G)$.  A fundamental result of Kashiwara and Schapira computes this Hom sheaf via sheaf operations \cite[Thm. 6.1.2]{kashiwara-schapira}:\footnote{More precisely, \cite{kashiwara-schapira} shows there is
 a morphism of this kind for $\mu sh^{pre}$ which is an isomorphism at stalks; the stated result follows upon sheafification.  
 See \cite{nadler-shende} for some detailed discussions about the sheafification of $\mu sh^{pre}$.} 
\begin{equation} \label{hom of microsheaves is muhom} 
\cH om_{\mu sh}(F, G) = \mu hom(F, G) := \mu_\Delta \Hom_{M \times M}(\pi_1^* F, \pi_2^! G). 
\end{equation}

For $F \in sh(M)$, one finds that the support of the image of $F$ in $\mu sh(\Omega)$ is $ss(F) \cap \Omega$.  For this reason, 
for any object $G \in \mu sh(\Omega)$, we sometimes write $ss(G)$ for the support of $G$. 

The sheaf $\mu sh_{T^*M}$ is conic, i.e. equivariant for the $\R^+$ scaling action.  In particular, $\mu sh_{T^*M}|_{T^\circ M}$ is locally constant
in the radial direction.  This being a contractible $\R^+$, we may define a sheaf of categories $\mu sh_{S^*M}$ on $S^*M$ equivalently by pushforward 
or pullback along an arbitrary section of $T^\circ M \to S^*M$. As $ss(G)$ is conic for $G \in \msh_{S^* M}$, it uniquely determine a set $ss^\infty(G) \subseteq S^* M$.

\begin{definition}\label{definition:microsheaves-cotangent}
For conic $K \subset T^*M$, let $\mu sh_{K} \subset  \mu sh_{T^*M}$ be the subsheaf of full subcategories on objects supported in $K$. Similarly, for  
$\Lambda \subset S^*M$, let $\mu sh_{\Lambda} \subset \mu sh_{S^*M}$ be the subsheaf of full subcategories on objects supported in $\Lambda$. 
\end{definition}

\begin{proposition}{\cite[Prop. 6.6.1]{kashiwara-schapira}} \label{existence of microstalk}
Let $M$ be a manifold and let $N \subset M$ be a submanifold. Let $L= T^*_NM$ and fix a point $p \in L$. Then there is an equivalence of categories: 
\begin{align} \label{stupid microstalk}
\cC &\xrightarrow{\sim} (\mu sh_{L})_p \\
A &\mapsto A_N, \nonumber
\end{align}
where $A_N$ is the image in $\mu sh$ of the constant sheaf on $N$ with value $A$.  
The corresponding result of course holds, and we use the same notations, in $S^*M$. 
\end{proposition} 

\begin{definition} \label{microlocally constructible} 
For $M$ real analytic, we define $\mu sh_{\R-c} \subset \mu sh$ as the subsheaf
of full subcategories on objects whose support is subanalytic and the closure of its smooth Lagrangian locus.  

For $M$ complex analytic, we define $\mu sh_{\C-c} \subset \mu sh_{\R-c}$ as the subsheaf of full subcategories on objects whose support is complex analytic and the closure of its smooth Lagrangian locus.
\end{definition}

Alternatively, one can also restrict to constructible sheaves all from the beginning and define
\begin{equation} \label{ks mush-R-c}
\mu sh_{T^*M, \R-c}^{pre}(\Omega) := sh_{\R-c}(M) / sh_{T^*M \setminus \Omega, \R-c}(M),
\end{equation}
and consider its sheafification. For a conic Lagrangian $\Lambda \subseteq T^* M$, define
\begin{equation}\label{ks mush-fixed}
\mu sh_{\R-c, \Lambda}^{pre}(\Omega) := sh_{\R-c, \Lambda \cup (T^* M \setminus \Omega)}(M) / sh_{\R-c, T^*M \setminus \Omega}(M). 
\end{equation}
Unwrapping the definition, we see that $ sh_{\R-c, \Lambda \cup (T^* M \setminus \Omega)}(M)= \{F \in sh_{\R-c}(M) | ss(F) \cap \Omega \subseteq \Lambda\}$. Similarly, one can define $\msh^\pre_{T^* M, \C-c}$ when $M$ is complex, and $\msh^\pre_{\Lambda, \C-c}$ when $\Lambda$ is complex.
By microlocal cut-off, these notions agree with the ones given previously.

\begin{lemma}{\cite[Theorem 3.2.2]{waschkies-microperverse}} \label{lem: msh_definitions-agree} 
The canonical map $\msh_{T^* M, \R-c}^\pre \rightarrow \msh_{T^* M, \R-c}$ induces an isomorphism upon sheafification. The similar statement holds for $\msh_{\R-c, \Lambda}$, $\msh_{T^* M, \C-c}$, and $\msh_{\C-c, \Lambda}$.
\end{lemma}

We will need the following lemma in the next section. Let $i: N \hookrightarrow M$ be an inclusion of closed submanifolds. The map $i$ induces an inclusion $i_\pi: T^*M|_N \hookrightarrow T^*M$ and the transpose of its derivative induces a projection 
\begin{align*}
(di)^t: T^* M|_N &\rightarrow T^* N\\
(y,\xi) &\mapsto (y, \xi \circ di_y).
\end{align*}

\begin{lemma} \label{lem: msh-on-sub}
The $*$-pushforward, $i_*: sh(N) \xrightarrow{\sim} sh_{T^* M|_N}(M)$, microlocalizes to an equivalence 
$$  [(di)^t]^* \mu sh_{T^* N} \xrightarrow{\sim} \mu sh_{T^* M|_N}.$$
\end{lemma}
(Here $\mu sh_{T^* M|_N} \subset \mu sh_{T^*M}$  is the subsheaf of full subcategories on objects supported in $T^*M|_N \subset T^*M$; see \Cref{definition:microsheaves-cotangent}.)
\begin{proof} 
  
Since $i$ is a closed embedding, \cite[Proposition 5.4.4]{kashiwara-schapira} implies that $ss(i_* F) = ({(di)^t})^{-1} ss(F)$ is contained in $T^* M |_N$ for $F \in sh(N)$. This implies that the assignment 
\begin{align*}
\mu sh_{T^* N}^{pre}(\Omega) &\rightarrow \left(({(di)^t})_* {i_\pi}^* \mu sh_{T^* M|_N}^{pre} \right)(\Omega)\\
F &\mapsto i_* F
\end{align*}
is well-defined. To check the induced map $({(di)^t})^* \mu sh_{T^* N} \rightarrow {i_\pi}^* \mu sh_{T^* M}$ on sheaves is an equivalence, one can check at stalks. Fully faithfulness is then implied by \cite[Proposition 4.4.7(ii)]{kashiwara-schapira}, as \eqref{hom of microsheaves is muhom} shows that the Hom is computed by $\mu hom$. Essential surjectivity is implied by \cite[Prop. 6.6.1]{kashiwara-schapira}.
\end{proof}

Before we leave this section, we mention two common tools for studying microsheaves. The first is the contact transformation.

\begin{theorem}[{\cite[Corollary 7.2.2]{kashiwara-schapira}}] \label{thm: real-contact-transform}
Let $\cU \subseteq S^* M$, and $\cV \subseteq S^* N$ be open sets, and $\chi: \cU \xrightarrow \cV$ be a contactomorphism. Then, for any given $p \in \cU$, shrink $\cU$ if needed, one can assume that there exists a sheaf $K \in sh(M \times N)$ such that the functor $\Phi_K: sh(M) \rightarrow sh(N)$ given by convolving with $K$ induces an equivalence 
\[ \Phi_K: \msh^{pre}_{S^* M} |_{\cU} \xrightarrow{\sim } \chi^* \left( \msh^{pre}_{S^* N}|_{\cV}\right). \]
Consequently, it induces an equivalence $\msh_{S^* M}|_{\cU} \xrightarrow{\sim} \chi^* \msh_{S^* N} |_{\cV}$ which commutes with the canonical map $\msh^{pre} \rightarrow \msh$.
\end{theorem}

The other one is the doubling trick, which states that, if $\cF \in \msh(\Omega)$, for some open set $\Omega \subseteq S^* M$, is a microsheaf whose support $ss^\infty(\cF)$ is contained in some closed Legendrian $\Lambda \subseteq S^* M$, then $\cF$ can be represented by a genuine sheaf $F$ (with possibly larger microsupport).

\begin{theorem}[{\cite[Theorem 7.18]{nadler-shende}, \cite[Theorem 4.47]{kuo-li-spherical} }] \label{thm: doubling}
Let $\Lambda \subseteq S^* M$ be a closed Legendrian and $\Omega \subseteq S^* M$ an open set. Fix a positive isotopy $\varphi_t$ defined near $\Lambda$  such that $\alpha(\dot{\varphi}_t)$ is nonzero on $\Lambda \cap \Omega$ and zero outside $\Omega$, and there exists $C(\Lambda) > 0$ such that $\varphi_t(\Lambda) \cap \Lambda = \varnothing$ on $\Omega$ for all $t \in [0,C(\Lambda)]$. Then, for $\epsilon \in (0,C(\Lambda))$, the composition 
 \[sh_{\Lambda \cup \Lambda_\epsilon}(M) \rightarrow \msh_{\Lambda\cup\Lambda_\epsilon}(\Omega) \rightarrow \msh_{\Lambda}(\Omega)\] 
is surjective, where $\Lambda_\epsilon \coloneqq \varphi_\epsilon(\Lambda)$ Here, we use the fact that $\Lambda$ and $\Lambda_\epsilon$ are disjoint in $\Omega$ for the second arrow. \footnote{A picture of $\Lambda \cup \Lambda_\epsilon$ can be found as the left-hand picture in \cite[Figure 3]{kuo-li-spherical}. Note the notation there differs from ours in that $\Omega$ is an open subset of $\Lambda$. This is no essential issue as $\msh_\Lambda$ is supported on $\Lambda$.}
\end{theorem}

\subsection{Microsheaves and Maslov index} \label{sec: maslov index}  In this subsection, we describe following \cite{chiu2017nonsqueezing} the sheaf quantization of a rotation of $J^1\mathbb{R}^n$. This material will be used explicitly only in \Cref{subsection:canonicalmicrostalkcomplex}.

Endow
\(T^*\mathbb{R}^n \times T^*\mathbb{R}\) with coordinates
\((x,\xi,t,\tau)\).  Identify the \(1\)-jet space
\(J^1\mathbb{R}^n\) with the hyperplane \(\{\tau=1\}\subset
T^*\mathbb{R}^n\times T^*\mathbb{R}\); the contact form
restricted to \(\{\tau=1\}\) is
\[
\alpha \;=\; \xi\cdot dx + dt.
\]

Fix $1 \leq \ell \leq n$. As in \cite{chiu2017nonsqueezing}, we consider a $1$-parameter family of contactomorphisms $\psi_t:J^1\mathbb{R}^n\to J^1\mathbb{R}^n$ 
\(\psi_t:J^1\mathbb{R}^n\to J^1\mathbb{R}^n\) given by
\[
\psi_t(x,\xi,\tau) \;=\; (x',\xi',\tau')
\]
with (for \(i=1,\dots,\ell\))
\[
\begin{aligned}
x'_i &= (\cos t)\, x_i - (\sin t)\,\xi_i,\\[4pt]
\xi'_i &= (\sin t)\, x_i + (\cos t)\,\xi_i,
\end{aligned}
\]
and for \(i=\ell+1,\dots,n\)
\[
x'_i = x_i,\qquad \xi'_i=\xi_i.
\]

The \(\tau\)-coordinate transforms by
\[
\tau' \;=\; \tau + F_{t}(x,\xi),
\]
where,
\begin{equation} \label{eq: action-functional}    
F_{t}(x,\xi) = \sum_{i=1}^\ell\left(\frac{\sin t\cos t}{2}\,(x_i^2-\xi_i^2)-2 \sin^2 t \cdot x_i\cdot \xi_i\right).
\end{equation}

Note that $\psi_t$ is the lift of a $1$-parameter family of sympectomorphisms of $T^*\mathbb{R}^n$ which rotates the $(x_i, \xi_i)$ coordinates clockwise for $1 \leq i \leq \ell$.  

These symplectomorphisms admit a sheaf quantization:  

\begin{proposition}[{\cite[Proposition 3.10]{chiu2017nonsqueezing}}\footnote{As the Hamiltonian is not compactly supported, the quantization does not follow directly from the usual result of Guillermou, Kashiwara, and Schapira \cite[Proposition A.6]{guillermou-kashiwara-schapira}.}] \label{prop: chiu-kernel}
Consider the space $\R^n \times \R^n \times \R_t \times \R_s$ where we view $s$ as the time direction. Then, there exists a sheaf kernel $\cS \in sh(\R^n \times \R^n \times \R_t \times \R_s)$ such that 
\begin{enumerate}
    \item $\cS|_{s = 0}=1_{\Delta_{\R^n}} \boxtimes 1_{[0,\infty)}$.
    \item $ss(\cS) \cap \{\sigma > 0\} \subseteq \Lambda_{\Psi}$ where we use $\sigma$ to denote the cotangent coordinate for $\R_s$ and $\Lambda_\Psi$ is the Legendrian lift of the isotopy movie of the Hamiltonian rotation.
    \item Over the open interval $s \in (0,\pi)$, we have the explicit formula,
    \[\cS|_{ (0,\pi)_s} = 1_{ \{(x,y,t,s)|t + F_{2s}(x,y) \geq 0 \}}.\]
\end{enumerate}
\end{proposition}

For a sheaf kernel $\cK \in sh(M \times N \times \R_t)$, one can consider the $\star$-convolution 
\begin{align*}
    sh(M \times \R_{t_1}) &\rightarrow sh(N \times \R_t) \\
    \cF &\mapsto \cK \star \cF \coloneqq {(p_{N,+})}_!( p_2^* \cK \otimes p_{M,1}^* \cF) 
\end{align*}
where $p_{M,1}$ and $p_2$ are the usual projections from $M \times N \times \R_{t_1} \times \R_{t_2}$ to 
$M \times \R_{t_1}$ and $M \times N \times \R_{t_2}$ but, when pushing forward, we use the addition for the $t$ direction by $p_{N,+}(x,y,t_1,t_2) = (y,t)$, $t = t_1 + t_2$ instead. 

By (2) of the above \Cref{prop: chiu-kernel}, $\cS \star (-)$ restricts to an equivalence on $\mu sh^{pre} |_{\{\sigma > 0\}}$ and hence to  the restriction of $\mu sh$ to $\{\sigma > 0\}$. 

Let \(K:=\{\xi=0,\ \tau=0\} = 0_{\mathbb{R}^n} \times \{0\} \subset J^1\mathbb{R}^n\) be the zero
section, and write \(p=(0,0,0)\) for the origin. Set \(K_t:=\psi_t(K)\) and note that \(K_{\pi/2}\) is the conormal to
\(C_k \coloneqq \{x_1=\cdots=x_k=0\}\subset\mathbb{R}^n\).

\begin{lemma}\label{lemma:fourier-shift-diagram}
We have the following diagram: 
\begin{equation}
\begin{tikzcd}
\mathcal{C}
  \arrow[rr, "A \mapsto A_L \boxtimes 1_{[0, \infty)}"] 
  \arrow[d, "{[-k]}"'] 
&& (\mu sh_K)_p 
  \arrow[d] 
\\
\mathcal{C} 
  \arrow[rr, "A \mapsto A_{C_k} \boxtimes 1_{[0, \infty)}"] 
&& (\mu sh_{K_{\pi/2}})_p
\end{tikzcd}
\end{equation}
where the right downward arrow is induced by the sheaf quantization of $\psi_{\pi/2}$ furnished by \Cref{prop: chiu-kernel}.
\end{lemma}

\begin{proof}
By the K\"unneth formula, it is sufficient to prove the case for $l=1$. Since we are rotating the Lagrangian by $90^o$, we set $s = \pi/2$ and we have $\cS|_{s = \pi/2} = 1_{ \{t - \frac{1}{2} x \xi \geq 0 \}}$ by (\ref{eq: action-functional}). Thus, seeing the commutativity amounts to computing the object
\[ \cS |_{s = \pi/2} \star (A_{\R^1} \boxtimes 1_{[0,\infty)_t}) = (p_{2,+})_! \left( A_{\{(x,y,t_1,t_2)| t_1 \geq 0, \ t_2 - \frac{1}{2}  \geq 0 \} } \right).  \]
Now, the projection $p_{2,+}$ can be decomposed to the naive projection $q(x,y,t_1,t_2) = (y,t_1,t_2)$  follows by the addition $a(y,t_1,t_2) = (y,t_1 + t_2)$, we can thus compute the pushforward in two steps. The claim is that $q_! \left( A_{\{(x,y,t_1,t_2)|t_1 \geq 0, \ t_2 - \frac{1}{2} x y \geq 0\} } \right) = A_{\{0\}}[-1] \boxtimes 1_{[0,\infty)_{t_1}} \boxtimes 1_{[0,\infty)_{t_2}}$, and we will be done because the projection formula will imply that 
\[a_!\left( A_{\{0\}}[-1] \boxtimes 1_{[0,\infty)_{t_1}} \boxtimes 1_{[0,\infty)_{t_2}} \right) = A_{\{0\}}[-1] \boxtimes (1_{[0,\infty)_{t_1}} \star 1_{[0,\infty)_{t_2})}) = A_{\{0\}}[-1] \boxtimes 1_{[0,\infty)_{t}} \]
as we desired. To see the claim, again by the projection formula, we can ignore $t_1$ and we only have to show that, if $f(x,y,t) \coloneqq (y,t)$, then $f_!( A_{\{t - \frac{1}{2} xy \geq 0\}}) = A_{(\{0,t)| t \geq 0\}} [-1] = A_{\{0\}}[-1]\boxtimes 1_{[0,\infty)}$. Lastly, we can check on stalks, and consider, for $(y,t) \in \R^2_{(y,t)}$,
\[ \left(f_! (A_{\{t - \frac{1}{2} xy \geq 0 \}}) \right)_{(y,t)} = \Gamma_c( \{x \in \R | t - (1/2) xy \geq 0 \}; A)   \]
the compactly supported sections on the set $\{x \in \R|t - \frac{1}{2} xy \geq 0\}$. We note that, when $y > 0$, $t - \frac{1}{2} xy$ is equivalent to $x \leq \frac{2t}{y}$ so for any pair $(y,t)$ the set $\{x \in \R|t - \frac{1}{2} xy \geq 0\}$ is topologically $[0,\infty)$ which has a vanishing compactly supported cohomology. The same situation holds when $y < 0$. Thus, the only non-trivial stalks live over $y=0$, in which case $x$ has no condition. In other words, the stalk is given by $\Gamma_c(\R;A) = A[-1]$ and is exactly what we wanted. 
\end{proof}

\begin{remark}
    \Cref{lemma:fourier-shift-diagram} can  presumably also be extracted from the  discussion in \cite[Sec.\ 7.5, Appendix A]{kashiwara-schapira}. Similar computations have also appeared in e.g.\  \cite{guillermou, jin-BO}.
\end{remark}

\subsection{Microsheaves on polarized (real) contact manifolds} 
In the previous section, we discussed microlocal sheaves on cotangent/cosphere bundles. The high-codimensional embedding trick of \cite{shende-microlocal} extends the definition of microlocal sheaves to arbitrary contact or exact symplectic manifolds equipped with Maslov data. In this section, we review how this works on polarized contact manifolds; we postpone the discussion of Maslov data to the next section.

The starting point of \cite{shende-microlocal} is the following lemma, which is a consequence of the functoriality of $\mu sh$ under (quantized) contact transformation \cite[Sec.\ 7]{kashiwara-schapira}.
\begin{lemma} \label{lem: thickening} 
Suppose given a contact manifold $V$ and a contact embedding $\iota: V \times T^*D^n \hookrightarrow S^* M$, for some disk $D^n$ and some manifold $M$. 
Then $\mu sh_{\iota(V \times D^n)}$ is locally constant along $D^n$, hence the pullback of a sheaf of categories
along $V \times D^n \to V$. 
\end{lemma} 

\begin{proof}
The statement is local, so we may assume $V$ is a Darboux ball. By the contact neighborhood theorem \cite[Theorem 2.5.15]{geiges2008introduction}, there is an open ball $U \subset S^*\mathbb{R}^{n+m}$ and a contactomorphisms $(V \times T^*D^n, V \times \{0\}) \simeq (U, U \cap S^*\mathbb{R}^n)$, where $S^*\mathbb{R}^m \subset S^*\mathbb{R}^{n+m}$ is induced by the standard inclusion $\R^m = \mathbb{R}^n \times \{0\}^n \hookrightarrow \R^{m + n}$. By the functoriality of $\mu sh$ under (quantized) contact transformation \cite[Sec.\ 7]{kashiwara-schapira}, we may therefore assume without loss of generality that $V = U \cap S^*\mathbb{R}^n$. But, by \Cref{lem: msh-on-sub}, we conclude that $\mu sh_{S^* \R^m \times 0_{\R^n}} = q^* \mu sh_{S^* \R^m}$ where $q: S^* \R^m \times T^* \R^n \rightarrow S^* \R^m$ is the projection.
\end{proof}

The basic idea of  \cite{shende-microlocal} is to use Lemma \ref{lem: thickening} as a definition of $\mu sh_V$.  
Let $(V, \xi)$ be a contact manifold.  Suppose given a (possibly positive-codimensional) contact embedding 
$\iota: V \to S^*M$ and any Lagrangian distribution $\eta$ of the symplectic normal bundle $N_V$ to $\iota$. The choice of $\eta$ comes with a coisotropic embedding $\tau_\eta: \operatorname{Op}(0_V \subseteq \eta) \to S^*M$, where $ \operatorname{Op}(0_V \subseteq \eta)$ is an unspecified open neighborhood of the zero section $0_V$ inside $\eta$. Furthermore, the differential 
\[d \tau_\eta: T_pV \times \eta_p \to T_{\tau_\eta(p,0)}S^*M\] satisfies the property that $d \tau_\eta(\{0\} \times \eta_p) = \eta_p$ for all $p \in V$. We denote by $V(\eta)$ the image of $\tau_\eta$ and call it a \emph{thickening} of $\iota(V)$ along $\eta$.

\begin{lemma}
The sheaf of categories $\mu sh_{V(\eta)}$ is the pullback of a certain sheaf of categories on $V$.  We will denote this sheaf as  $\mu sh_{V,\iota,\eta}$. 
\end{lemma}

\begin{proof}
The Lagrangian bundle structure $\eta$ provides, near $V$, a contraction $r: V(\eta) \rightarrow V$, and thus defines a sheaf $\mu sh_{V, \iota, \eta} \coloneqq r_* \mu sh_{V(\eta)}$ by pushforward. However, the map $r$ is locally given by the projection $V \times D^n \rightarrow V$, and we've shown in \Cref{lem: thickening} that the sheaf $\mu sh_{V(\eta)}$ is constant along the fiber direction, so $\mu sh_{V(\eta)}$ can be recovered from $\mu sh_{V,\iota,\eta}$ by
$$\mu sh_{V(\eta)} = r^* r_* \mu sh_{V(\eta)} = r^* \mu sh_{V,\iota,\eta}.$$
This completes the proof. 
\end{proof}

To eliminate the dependence on $\iota$ and $M$, one notes that Gromov's $h$-principle for
contact embeddings implies the existence
of high codimension embeddings of $V$ into the standard contact $\R^{2n+1}$ for
large enough $n$, the space of which moreover becomes arbitrarily connected
as $n \to \infty$.   There remains the (continuous) dependence on a choice of 
polarization of the stable symplectic normal bundle (``stable normal polarization'').  That is:

\begin{theorem} \label{normally polarized contact microsheaves} 
\cite{shende-microlocal}
Given a contact manifold $(V, \xi)$ and 
a polarization
$\eta$ of the stable symplectic normal bundle to $V$, 
there is a canonical sheaf of categories $\mu sh_{V, -\eta}$ on $V$, locally isomorphic
to $\mu sh_{S^*M}$ for any $M$ of the appropriate dimension.
\end{theorem}

Above, the sign $-$ on $-\eta$ can be regarded as just a notational choice.  To give it an actual meaning, 
note that $BO \to BU$ is a morphism of spectra, so given a stable symplectic vector bundle $E$ 
on some topological space $X$, 
say classified by some map $E: X \to BU$, and a polarization of $E$, i.e. lift to some $F: X \to BO$, 
then $-F$ gives a polarization of $-E$, where $-$ is the pullback by the canonical `inverse' involution 
on $BO$ or $BU$.

In fact \cite[Sec.\ 10.2]{nadler-shende}, it is always possible to build a contact manifold, $V^{T^*LGr(\xi)}$, which is a 
bundle over $V$ with fibers the cotangent bundles to the Lagrangian Grassmannians
of the contact distribution $\xi$.  The fiberwise zero section is the Lagrangian Grassmannian
bundle $V^{LGr(\xi)}$.  
The virtue of $V^{T^*LGr(\xi)}$ is that it has a canonical polarization of  the contact distribution, hence a canonical stable
normal polarization.
Thus we define $\mu sh_{V^{T^*LGr(\xi)}}$, and, restricting supports, $\mu sh_{V^{LGr(\xi)}}$.
It is evident from the construction that $\mu sh_{V^{LGr(\xi)}}$ is locally constant in the Lagrangian Grassmannian direction.
We may also stabilize $\xi \to \xi \oplus T^*\R^n$; taking $n \to \infty$, we have a sheaf of categories $\mu sh_{V^{LGr}}$ on the (stable) Lagrangian Grassmannian
bundle $V^{LGr}$, locally constant along the Lagrangian Grassmannian direction. 

Let us explain how this recovers the previous notion. 
A polarization $\rho$ of the contact distribution is (by definition) a 
section $\rho: V \to V^{LGr(\xi)}$.  Now, the normal direction of a neighborhood to $\rho(V) \subset V^{T^*LGr(\xi)}$  can be identified with $T^*LGr(\xi)$. This carries the canonical 
polarization by the cotangent fiber.  We may combine this with 
canonical polarization of the symplectic stable normal bundle of 
$V^{T^*LGr(\xi)}$ to obtain a polarization of the symplectic stable normal bundle of $V$. 
It's an exercise to see that this polarization is canonically identified with $-\rho$.  
We conclude:

\begin{lemma} \label{tangent normal polarization comparison} 
For $\rho$ a polarization of the contact distribution, 
there is a canonical isomorphism $\rho^* \mu sh_{V^{LGr(\xi)}} = \mu sh_{V, \rho}$. 
\end{lemma}
\begin{proof}
We may choose the embedding of $V$ by first embedding $V^{LGr(\xi)}$.  
To obtain $\mu sh_{V, \rho}$, we must then thicken $V$ along a polarization of its normal bundle which stabilizes to $-\rho$. 
We obtain such by thickening $V$ along $V^{LGr(\xi)} \subset V^{T^*LGr(\xi)}$ and then along the canonical polarization
of the normal bundle to $V^{T^*LGr(\xi)}$.  But the same procedure defines $\mu sh_{V^{LGr(\xi)}}$. 
\end{proof} 

 In case $V = S^*M$, there is a polarization $\phi$ given by the cotangent fiber.
 Now we have two notions of $\mu sh_{S^* M}$: the original given by \eqref{ks mush}, and then the construction of \Cref{normally polarized contact microsheaves}. 

\begin{lemma} \label{stupid normal polarization} 
Let $M$ be a manifold, and $\nu$ the stable normal bundle of $M$.  Then the stable symplectic
normal bundle to $S^*M$ is $\nu \oplus \nu^*$, and the stable normal polarization by 
$\nu$ is canonically identified with $-\phi$.  
\end{lemma}

\begin{proposition}
  \label{prop: mush mush comparison}
There is a canonical equivalence of categories $\mu sh_{S^*M} \simeq \mu sh_{S^*M, \phi}$, where the LHS is defined by 
\eqref{ks mush} and the right hand side by Theorem \ref{normally polarized contact microsheaves}.   
\end{proposition}
\begin{proof}
Said differently, we officially define $\mu sh$ by embedding into $J^1 \R^n \subset S^* \R^{n+1}$, but
the construction makes sense for any embedding into any cosphere bundle, e.g. the embedding of $S^*M$ into itself. 
We should check these give the same result. 

So embed $i: M \hookrightarrow \R^n$; we denote the normal bundle $\nu$; note this is a representative of the stable normal
bundle.  Let $\lambda_{T^*\R^n}$ and 
$\lambda_{T^*M}$ be the canonical one forms on cotangent bundles. 
Any splitting $\sigma$ of the vector bundle map
$T^* \R^n|_M \twoheadrightarrow T^*M$ will satisfy $\sigma^* \lambda_{T^*\R^n} = \lambda_{T^*M}$, hence
define a contact embedding $\sigma: S^*M \hookrightarrow S^* \R^n$. 
The symplectic normal bundle to $\sigma$ is the restriction of $T^*_{T^*M} T^*\R^n = \nu \oplus \nu^*$.
Let $S^*M(\nu)$ be a thickening of $\sigma(S^*M)$ in the direction $\nu$. 
By definition, $\mu sh_{S^*M, -\nu} = \sigma^* \mu sh_{S^*M(\nu)}$.  
\Cref{stupid normal polarization} gives $\phi = -\nu$. 
However, we've seen by the projectivized version of \Cref{lem: msh-on-sub} that $i_*: sh(M) \to sh(\R^n)$ microlocalizes to define an isomorphism 
$i_* :\mu sh_{S^*M} \xrightarrow{\sim} \sigma^* \mu sh_{S^*M(\nu)}$.
\end{proof} 

Consider now a Legendrian $L\subset V$.  Locally near $L$, we may choose a Weinstein neighborhood 
$D^* L \hookrightarrow V$, on which we may consider the polarization $\phi_L$ by cotangent fibers.  Note that if $V = S^*M$, the polarization $\phi_L$ has nothing to do with the fiber polarization of $S^*M$. 

\begin{lemma} \label{lem: polarization-local-systems} 
One can uniformly fix trivializations $\mu sh_{L, \phi_L} \simeq loc_L$. 
\end{lemma}

\begin{proof}
We may assume without loss of generality that $V=T^*L$. To compute $\mu sh_{L, \phi_L}$, we note that the jet bundle $J^1 L$ can be viewed as an open subset of the cosphere bundle $S^*(L \times \R)$ by the contact embedding $J^1 L \hookrightarrow S^* (L \times \R), (x,\xi,t) \mapsto (x,t,[\xi,1]).$ Since $\R$ is contractible, $\phi_L$ on $J^1 L$ is the same as the restriction of $\phi_{L \times \R}$ on $S^*(L \times \R)$ along this embedding. Thus, we have the identification between sheaves of categories $\mu sh_{L, \phi_L} = \mu sh_{L \times T^*_{0,>} \R, \phi_{L \times \R} }$ on $S^*(L \times \R)$.  By \Cref{prop: mush mush comparison}, this realizes $\mu sh_{L,\phi_L}$ as $\mu sh_{S^*(L \times \R); 0_L \times T_0^* \R}$ in the classical sense recalled in \eqref{ks mush}. But the map
\begin{align*}
    loc_L &\xrightarrow{\sim} \mu sh_{S^*(L \times \R); 0_L \times T_0^* \R}\\
    l &\mapsto l \boxtimes 1_{[0,\infty)}
\end{align*}
 identifies it with local systems on $L$.
 \end{proof}

In fact, we have made a choice in the above Lemma --  the space of trivializations for a given $L$ is a $\mathrm{H}^0(L, \Z)$ torsor.  Our trivialization above is uniform in the sense that it is compatible with open embeddings and, in an appropriate sense, with stabilization.  Even still, the space of uniform trivializations choices  is still naturally only a $\Z$-torsor, of which we have chosen some particular element.  For a discussion of this point, see \cite[Remark 4.32]{Gammage-Shende-large-volume}.

\subsection{Maslov data} \label{sec: mush}
Because $\mu sh_{V^{LGr}}$ 
is locally constant along the Lagrangian Grassmannian direction, one may expect that its descendability from $V^{LGr}$ to $V$ depends only on the `monodromy' in this direction.  
Indeed this is the case, as was established in \cite{nadler-shende}; we will recall the setup here. Recall that for a symmetric monoidal category $\cC$, the group of invertible objects is denoted $\Pic(\cC)$.  

\begin{theorem}[{\cite[Sec. 11]{nadler-shende}}] \label{mush} There is a map of infinite loop spaces
$\mu_{\cC} : LGr \to B\Pic(\cC)$ such that the sheaf of categories $\mu sh_{V^{LGr}}$ descends to the $B\Pic(\cC)$ bundle over $V$ classified by
the map 
$$V \xrightarrow{\xi} BU \to BLGr \xrightarrow{B \mu_{\cC}} B^2 \Pic(\cC) $$
\end{theorem}

\begin{definition} \label{def:maslov}
By $\cC$-Maslov data for $V$, we mean a null-homotopy of the map $V \to B^2 Pic(\cC)$. 
By a $\cC$-grading, we mean a null-homotopy of $V \xrightarrow{B \mu_{\cC}} B^2 \Pic(\cC) \to B^2 \pi_0 \Pic(\cC)$. We refer to the space of $\cC$-Maslov data lifting a given $\cC$-grading as $\cC$-orientation data.  
\end{definition} 

In the case when $\cC = R -mod$ for a commutative ring (spectrum) $R$, we often simplify the notation by writing $R$ in place of $R-mod$, e.g. $\mu_{R} := \mu_{R-mod}$,  $\Pic(R) \coloneqq \Pic(R -mod)$, etc.

A polarization $\rho$ provides a null-homotopy of the map $V \xrightarrow{\xi} BU \to BLGr$, so 
$B\mu_C \circ \rho$ is $\cC$-Maslov data.  Lemma \ref{tangent normal polarization comparison} implies that
$\mu sh_{V, \rho} = \mu sh_{V, B\mu_C \circ \rho}$, where the left hand side is defined as in 
Theorem \ref{normally polarized contact microsheaves} and the right hand side is understood
in the sense above.   For this reason, given a polarization $\rho$ we will also just write $\rho$ for the Maslov
data $B\mu_C \circ \rho$.

\begin{definition} \label{def: microsheaves}
For $V$ a contact manifold, and a $\cC$-Maslov datum $\eta$ for $V$, we write 
write $\mu sh_{V, \eta}$ for the sheaf of categories on $V$  
obtained by the pullback along the zero section of the $B\Pic(\cC)$ bundle, trivialized by $\eta$, to which \Cref{mush}
asserts that $\mu sh_{V^{LGr}}$ descends.  

For a subset $X \subset V$ 
we write $\mu sh_{X, \eta} \subset \mu sh_{V, \eta}$ for the sheaf of full subcategories on objects supported in $X$.  For an exact (real) symplectic manifold $(W, \omega = d\lambda)$, we always use implicitly the embedding in the contactization $W = W \times \{0\} \subset W \times \mathbb{R}$, and hence  write $\mu sh_{W, \eta}:= \mu sh_{W \times \{0\}, \eta}|_W$. 
\end{definition}

Thus $\mu sh_{V, \,\cdot\,}$ is a map from the space $Mas(V)$ of Maslov data for $V$ to the category $sh(V, \cC-cat)$ of sheaves of $\cC$-linear categories on $V$.  In particular, a homotopy of $\cC$-Maslov data $h: \mu \sim \nu$ induces an equivalence of sheaves of categories $\psi(h): \mu sh_{V, \mu} \cong \mu sh_{V, \nu}$, and a homotopy $g: h_1 \approx h_2$ induces an invertible natural transformation between the equivalences $\psi_{h_1}, \psi_{h_2}: \mu sh_{V, \mu} \cong \mu sh_{V, \nu}$. Taking based loops at some Maslov datum $\eta$, we get a map $\Omega_\eta Mas(V) \to Aut_C(\mu sh_{V, \eta})$.  Now, the space of Maslov data is a torsor for $\Map(V, B\Pic(\cC))$; if this is nonempty, it follows that
$\Omega_\eta Mas(V) = \Map(V, Pic(\cC))$.  It follows from the construction in \cite[Definition 11.18]{nadler-shende} that the map
$\Omega_\eta Mas(V) \to Aut(\mu sh_{V, \eta})$
is the natural map $\Map(V, Pic(\cC)) \to Aut_C(\mu sh_{V, \eta})$. In particular, as $\Map(V, Pic(\cC))$ classifies invertible local system on $V$, any homotopy $g: h_1 \approx h_2$ as above corresponds to some invertible $l(g) \in loc(V)$ and the equivalences are related by
\[ \psi_{h_2}(-)= l(g) \otimes \psi_{h_1}(-): \msh_{V,\mu} \cong \msh_{V,\nu}. \]

\begin{corollary} \label{cor: noncanonical}
     Let $V$ be a contractible contact manifold, and $\mu, \nu$ any two choices of $\cC$-Maslov data for $V$ inducing the same $\mathcal{C}$-grading.  Then there is an isomorphism, canonical up to non-canonical natural transformation,  $\mu sh_{V, \mu} \cong \mu sh_{V, \nu}$ 
\end{corollary}
\begin{proof}
    Similar to Maslov data, the space of gradings is a torsor for $\Map \left(V, \tau_{\leq 1} B\Pic(\cC) \right)$; %\sayVS{I'm not sure this is right} \sayCK{I suppose the concern is that $\tau_{\leq 1}$ being a left adjoint doesn't commute with $\Map(V,-)$? In any case, I modified it to the closest thing that I believe it's correct.} 
    here the hypothesis of `inducing the same $C$ grading' should be understood as meaning that we are given a  choice of path between the gradings associated to the given Maslov data.  Since $V$ is assumed contractible, we may lift the path to a path of Maslov data and obtain the desired isomorphism.  Two different paths differ by an loop which is trivial in $\Map(V, \pi_0(Pic(\cC)))$ hence by an automorphism which is the identity in 
    $\pi_0(Aut_C(\mu sh_{V, \eta}))$.
\end{proof}

In the discussion thus far, we have been agnostic as far as the choice of the category $\cC$, and we have also not needed
to compute the map $\mu_\mathcal{C}: U/O \to B\Pic(\cC)$.  We now turn to this question. 
Note first that given a map of symmetric monoidal stable categories $\cC \to \cD$, it follows from the construction 
that $\mu_{\cD}$ is the composition of $\mu_{\cC}$ with the natural map $Pic(\cC) \to Pic(\cD)$.  
In particular, when $R$ is a discrete commutative ring ($R = \pi_0(R)$), the map $\mu_{R-mod}$ factors through $\mu_{\Z}$.  
The map $\tau_{\le 0} \Omega \mu_{\Z}:\Omega LGr \to \Z$ was shown to be the Maslov index by Kashiwara and Schapira \cite[Thm. 7.5.11]{kashiwara-schapira}, and was later fully characterized by Guillermou \cite{guillermou}.  
More generally, any symmetric monoidal stable category $C$ admits a symmetric monoidal functor from the category of spectra (aka modules over the sphere spectrum $\S$).   The map $\mu_{\S}: \Omega LGr \to \Pic(\S)$ was shown by Jin
\cite{jin-J} to agree with the J-homomorphism.
By truncation one recovers Guillermou's result in a more convenient (for us) formulation. 

\begin{theorem}{\cite{guillermou, jin-J}} \label{thm: guillermou-maslov} 
The map $\Omega \mu_{\Z}: \Omega LGr \to \Pic(\Z)$ is the following composition: 
$$\Omega LGr \to  \tau_{\le 1}(\Omega LGr) \xrightarrow{\tau_{\le 1}(J)} \tau_{\le 1} \Pic(\S) = \Pic(\Z).$$
\end{theorem}

It follows in particular that $\Z$ Maslov data is precisely grading/orientation data, and hence that grading/orientation data provide $R$ Maslov data for any 
commutative ring $R$ (although  not all $R$ Maslov data need arise in this way).

\subsection{Secondary Maslov data} \label{secondary maslov data}   
If $L \subset V$ is a Lagrangian (resp. Legendrian) in a symplectic (resp. contact) manifold, as in \Cref{secondary Z-Maslov}, the Weinstein neighborhood theorem provides a polarization $\phi_L$ near $L$. Assume further that $V$ is equipped with a Maslov datum $\eta$.  Recall from \cite{nadler-shende} that we say a homotopy $\phi_L \sim \eta|_L$ is a secondary Maslov datum for $L$.   

As noted in \cite[Remark 11.20]{nadler-shende}, a choice of secondary Maslov data identifies microsheaves on $L$ with local systems. Indeed, a secondary Maslov datum induces an equivalence $\mu sh_{L, \eta} \simeq \mu sh_{L, \phi_L}$ which we can further compose with the equivalence $\mu sh_{L, \phi_L} \simeq loc_L$ from \Cref{lem: polarization-local-systems}. Taking stalks at a smooth point $p \in L$, we obtain:  

\begin{corollary}\label{corollary:microstalk-secondarymaslov}  
Fix a contact manifold (resp.\ exact symplectic manifold) $V$ and a Legendrian (resp.\ conical Lagrangian) $L \subset V$, and $p \in L$. Fix a Maslov datum on $V$ and a secondary Maslov datum on $L$. Then there is an equivalence
\begin{equation}\label{equation:microstalk-secondary}
    \mathcal{C} \xrightarrow{\sim} (\mu sh_L)|_p.
\end{equation}
Any choice of such an equivalence is termed a microstalk functor. 
\end{corollary}

The microstalk functor \eqref{equation:microstalk-secondary} depends on the choice of secondary Maslov datum $\eta \sim \phi_L$, and the ambiguity is a torsor for $\Pic(\cC)$. However, if $L$ is endowed with the additional datum of a \emph{secondary $\cC$-grading}, we can cut down the ambiguity to a smaller group.

To explain this, let $\eta, \eta'$ be $\cC$-Maslov data on a contact manifold $V$ (henceforth we restrict our attention to the contact case, leaving the symplectic analogue to the reader). As noted in the proof of \Cref{spin polarization}, a homotopy of Maslov data $\eta \sim \eta'$ is the same thing as null-homotopy of the difference $[\eta, \eta']: V \to \Omega B^2 Pic(\cC)$. Similarly, if $\eta, \eta'$ are $\cC$-gradings (\Cref{def:maslov}), then a homotopy of $\cC$-gradings is a homotopy of the corresponding map $V \to \Omega B^2\pi_0 \Pic(\cC)$. 

Suppose now that $L \subset V$ is a Legendrian, and let $\phi_L$ denote the (Maslov datum induced by the) canonical fiber polarization near $L$. Consider the difference $[\eta, \phi_L]: Op(L) \to \Omega B^2\Pic(\cC)$, and let $\overline{[\eta, \phi_L]}: Op(L) \to \Omega B^2\pi_0 \Pic(\cC)$ denote the composition of $[\eta, \phi_L]$ with the natural map $\Omega B^2 \Pic(\cC) \to \Omega B^2 \pi_0 \Pic(\cC)$. A secondary $\cC$-grading $h$ shall mean a null-homotopy of $\overline{[\eta, \phi_L]}.$

\begin{definition}
    Given $V, L, h$ as above, a secondary $\cC$-orientation datum for $L$ is a is a null-homotopy of $[\eta, \phi_L]: Op(L) \to \Omega B^2\pi_0 \Pic(\cC)$ lifting the null-homotopy $h$.
\end{definition}

The difference between any secondary $\cC$-orientation data lifting $h$ is a map $V \to \Omega^2 B^2 \Pic(\cC)= \Pic(\cC)$ such that the composition $V \to \Omega^2 B^2 \Pic(\cC) \to \Omega^2 B^2 \pi_0 \Pic(\cC)$ is null, i.e.\ a map $V \to \Omega^2B^2 \Pic_0(\cC)= \Pic_0(\cC)$. 
\begin{corollary}\label{corollary:microstalk-first}
    In the situation of \Cref{corollary:microstalk-secondarymaslov}, assume given a secondary $\cC$-grading $h$ on $L$. Given two choices of secondary orientation data lifting $h$, the corresponding microstalk functors \eqref{equation:microstalk-secondary} are a torsor for $\Pic_0(\cC)$. 
    
    Since the space of secondary orientation data is certainly nonempty when $L$ is contractible, \eqref{equation:microstalk-secondary} is canonical up to an action of $\Pic_0(\cC)$ by natural transformation. In particular, \eqref{equation:microstalk-secondary} is unambiguous at the level of \emph{objects}.
\end{corollary}

\subsection{Constrained Maslov data} \label{sec: contrained}  
In this subsection, we define a notion of \emph{constrained} Maslov data. The purpose of the discussion here is to have a framework for perverse t-structures that come from exotic t-structures on the coefficient category $\cC$ that are different from the standard one $R -mod$ as considered in \Cref{def: exotic-microperverse-t-structure}. None of this material is needed for proving the main results of this paper stated in the introduction and readers who care primarily about the canonical t-structure can safely skip it.

Let $\{\cD_i\}$ be a collection of subcategories $\cD_i \subseteq \cC$. One denote by \[\Pic(\cC)_{\{\cD_i\}} \coloneqq \{x \in \Pic(\cC)| x \otimes y_i \in D_i, \ \forall y_i \in D_i \forall i \}  \] the submonoid of $\Pic(\cC)$ which preserves each $\cD_i$. When the collection consists of only one subcategory $\cD$, we simplify the notation and denote it by $\Pic(\cC)_\cD$.

Beware that in general, elements of $\Pic(\cC)_{\{\cD_i\}}$ need not induce an automorphism on each $\mathcal{D}_i$. 

\begin{example} \label{eg: only-monoid}
Suppose $\cC = R -mod$ for a discrete ring $R$ and $\cC^{\geq 0}$ is the subcategory of objects supported in non-negative degrees. Then tensoring with $R[-1] \otimes M = M[-1]$ has the effect of shifting the cohomology degree up by one so $R[-1] \in \Pic(R)_{\cC^{\geq 0}}$. But clearly, its inverse $R[1] \not \in \Pic(R)_{\cC^{\geq 0}}$ so $\Pic(R)_{\cC^{\geq 0}}$ is only a monoid instead of a group.   
\end{example}

\begin{definition} 
    We say  $\{\cD_i\}$ is an anchored collection if $\Pic(\cC)_{\{\cD_i\}}$ is a subgroup. 
\end{definition}

%The main examples of anchored collections come from $t$-structures. 
Our reason for considering anchored collections is that the situation encountered in \Cref{eg: only-monoid} (where the functor induced by tensoring with $R[-1]$ fixes $\{\cC^{\geq 0}\}$ but is not surjective) cannot happen:

\begin{lemma} \label{lem: anchored-surjection}
Let $\{\cD_i\}$ be an anchored collection. Then, for any $x \in \Pic(\cC)_{\{\cD_i\}}$, the functor $x \otimes (-): \cC \xrightarrow{\sim} \cC$ restricts to an equivalence $x \otimes (-): \cD_i \xhookrightarrow{\sim} \cD_i$.    
\end{lemma}

\begin{proof}
Fully-faithfulness is automatic so we show that if $x \in \Pic(\cC)_{\{\cD_i\}}$, then for any $a \in \cD_i$ there exists $b \in \cD_i$ such that $x \otimes b = a$ . But we know $b$ exists in $\cC$, and it must satisfy $b = x^{-1} \otimes a$. The fact that $b \in \cD_i$ then follows from the anchored assumption since $x^{-1}$ is also in $\Pic(\cC)_{\cD_i}$.
\end{proof}

From now on, we always assume $\{ \cD_i\}$ to be anchored. Clearly, the identity component $\Pic(\cC)_0 \subseteq \Pic(\cC)_{\{\cD_i\}}$ is contained in any such subgroup. Thus, the cofiber $\Pic(\cC)_{\{\cD_i\}} \rightarrow \Pic(\cC)$ is a discrete group, as it is a quotient of $\pi_0(\Pic(\cC))$. We denote it by  $\pi_0 \left(\Pic(\cC); \{ \cD_i\}\right)$.

\begin{definition} \label{def: secondary-orientation}
A $(\cC,\{\cD_i\})$-grading is a null-homotopy of the map 
\[V \rightarrow B^2 \Pic(\cC) \rightarrow \pi_0 \left(\Pic(\cC); \{ \cD_i\}\right).\] 
A $(\cC,\{\cD_i\})$-orientation of a given $(\cC,\{\cD_i\})$-grading is a lift to a $\cC$-Maslov datum.
\end{definition}

\begin{example}\label{eg: coherent-coefficient}
Let $\cC = \Coh(\P^1)$ be the category of coherent sheaves on $\P^1$. In this case, invertible objects are of the form $\cO(n)[m]$ for $n, m \in \Z$, so $\pi_0(\Pic(\P^1)) = \Z \times \Z$. Let $\mathcal{D}_1$ be the full subcategory whose objects are honest coherent sheaves (viewed as complexes supported in degree zero). Then the group $\Pic(\P^1)_{\{\mathcal{D}_1\}}$ contains only $\cO(n)$ for $n \in \Z$. Thus $\pi_0 \left(\Pic(\cC); \{ \cD_1\}\right) = \Z$ only remembers the homological degree shift.
\end{example}

As illustrated in the diagram below, a usual $C$-grading induces a $(\cC,\{\cD_i\})$-grading, for which a $\cC$-orientation induces a $(\cC,\{\cD_i\})$-orientation. However, the latter is more general. 

\begin{equation}\label{equation:chris-big-diagram}
\begin{tikzcd}[column sep=large, row sep=large]
& B^2\Pic(\mathbb{S})_0 \arrow[r] \arrow[d] 
& B^2\Pic(\mathcal{C})_0 \arrow[r] \arrow[d] 
& B^2\Pic(\mathcal{C})_{\{\mathcal{D}_i\}} \arrow[d] \\[0.6em]
BU \arrow[r] \arrow[ur, dashed] 
& B^2\Pic(\mathbb{S}) \arrow[r] \arrow[d] 
& B^2\Pic(\mathcal{C}) \arrow[r, equals] \arrow[d] 
& B^2\Pic(\mathcal{C}) \arrow[d] \\[0.6em]
& B^2\mathbb{Z} \arrow[r] 
& B^2\!\left(\pi_0\Pic(\mathcal{C})\right) \arrow[r] 
& B^2\!\left(\pi_0 \big(\Pic(\mathcal{C});\{\mathcal{D}_i\}\big)\right)
\end{tikzcd}
\end{equation}

We can give ``constrained'' analogues of the constructions in \Cref{secondary maslov data}. Namely, suppose now that $L \subset V$ is a Legendrian. If $\{\cD_i\}$ is a collection of subcategories, as in \Cref{def: secondary-orientation}, a \emph{secondary orientation on $L$ constrained by $\{\cD_i\}$} is a homotopy between the polarization $(\cC,\{\cD_i\})$-grading and the $(\cC,\{\cD_i\})$-grading induced from $\eta$.

A secondary Maslov datum for $L$ lifting a given secondary polarization constrained by $\{\cD_i\}$ shall be called a \emph{secondary orientation datum for $L$ constrained by $\{\cD_i\}$}.

\begin{corollary}[cf.\ \Cref{corollary:microstalk-first}] \label{corollary: reduced-ambiguity}
Let $\eta$ be a Maslov datum and $L$ be an Legendrian endowed with a secondary grading constrained by $\{\mathcal{D}_i\}$. Then the ambiguity of the equivalence 
$\cC \cong (\mu sh_{L,\eta})_p $ from \Cref{corollary:microstalk-secondarymaslov} can be reduced to $\Pic(\cC)_{\{\cD_i\}}$. 
\end{corollary}
In particular, per \Cref{corollary: reduced-ambiguity}, the statement that ``the microstalk of $\cF \in (\mu sh_{L, \eta})|_p$ is contained in $\cD_i$" is well-defined.

\begin{example}
Suppose that $(\cC, \cD)$ are as in \Cref{eg: coherent-coefficient}. Then the microstalk of  $\cF \in (\mu sh_{L, \eta})|_p$ is an object in $\Coh(\P^1)$ which is well-defined up to tensoring with $\cO(n)$. Hence, it is meaningful to ask whether the microstalk of $\cF$ belongs to the heart $\Coh(\P^1)^\heartsuit$ (i.e.\ whether it is represented by an honest coherent sheaf).
\end{example}

\section{Microsheaves in the complex setting} \label{complex microsheaves} 

\subsection{Microsheaves on complex cotangent bundles} \label{sec: complex-cotangent-caes}

We now review the results of Waschkies \cite{waschkies-microperverse}. Denote by $\pi: T^\circ M \rightarrow \P^* M$ the projection. The perverse microsheaves on $\P^* M$ \cite[Definition 6.1.2]{waschkies-microperverse} is defined as a subsheaf of the following sheaf. 

\begin{definition} \label{def: P-mush-cotangent}
We define the presheaf $\Pmsh_{\P^* M}^\pre$ on $\P^* M$ by \[\Pmsh_{\P^* M}^\pre(\Omega) \coloneqq \msh^\pre(\pi^{-1}(\Omega)) \]
for an open set $\Omega \subseteq \P^* M$, and denote its sheafification by $\Pmsh_{\P^* M}$. If $\Lambda \subseteq \Omega$ is a closed complex Legendrian, we can similarly define a presheaf on $\Omega$, using the notation in \eqref{ks mush-fixed}, by
\[\Pmsh_{\Lambda}^\pre(\Omega') \coloneqq \msh_{\pi^{-1}(\Lambda)}^\pre(\pi^{-1}(\Omega)), \]
and denote its sheafification by $\Pmsh_\Lambda$. Per \Cref{lem: msh_definitions-agree}, the canonical map  $\Pmsh_\Lambda \hookrightarrow \Pmsh$ is fully faithful and its image consists of those objects supported in $\Lambda$. Lastly, we define $\Pmsh_{\C-c}$ to be the subsheaf consisting of objects supported on complex Legendrians, which can also be obtained as the sheafification of $\Pmsh_{\C-c}^\pre(\Omega) = \msh_{\C-c}^\pre(\pi^{-1}(\Omega))$
\end{definition}

The following \Cref{thm: waschkies-sh-quo-loc} is one of the main theorems in \cite{waschkies-microperverse}, which gives a simple description of $\Pmsh_\Lambda$ at stalks when $\Lambda$ is in generic position:

\begin{definition}\label{def: generic-position}
Suppose that $M$ is complex analytic and $\Lambda \subset P^*M$ is a (singular) complex Legendrian. Let $r: P^*M \to M$ be the projection. We say that $\Lambda$ is in generic position if that $m \in M$ has the property that $\Lambda \cap r^{-1}(m)= \{p_1, \dots, p_k\}$ is a finite set. 
\end{definition}

\begin{theorem}[{\cite[Thm.\ 5.1.5]{waschkies-microperverse}}] \label{thm: waschkies-sh-quo-loc}
In the situation of \Cref{def: generic-position}, let $p \in r^{-1}(m) \cap \Lambda$.
There is a fully faithful functor 
\[(\Pmsh_\Lambda)_p = \mu sh^{pre}_{\pi^{-1}(\Lambda)}(\mathbb{C}^* \cdot p) \to (sh_{\C-c}/loc)_{m}= (sh_{\C-c})_m/loc_m.\] Its essential image is the full subcategory of $(sh_{\C-c}/loc)_{m}$ consisting of objects with the following properties: There exists an  open neighborhood $Op(\mathbb{C}^* \cdot p)$ depending only on $\Lambda$ and $p$ such that the microsupport of the objects inside $Op(\mathbb{C}^* \cdot p)$ is contained in $\Lambda \cap Op(\mathbb{C}^* \cdot p)$.
\end{theorem}

\begin{proof}
The fully faithful functor is furnished by Waschkies \cite[Thm.\ 5.1.5]{waschkies-microperverse} using the microlocal cutoff. To recall its property, as the description is local, we may assume $M$ is a ball, so that we have coordinates $T^\circ M= M \times T^\circ _m M$. \cite[Thm.\ 5.1.3(3)]{waschkies-microperverse} implies that there exists a small neighborhood $U$ containing $m$ and an open cone $\gamma$ containing $\mathbb{C}^*\cdot p$ such that Waschkies' functor $(\Pmsh_\Lambda)_p \rightarrow (sh_{\C-c})_m/loc_m$ lands inside the subcategory consisting of objects whose microsupport inside $U \times \gamma$ is contained in $\Lambda \cap (U \times \gamma)$. Let us temporarily denote this subcategory by $\mathcal{A}$.

We should prove that every $F \in \mathcal{A}$, i.e.,  every $F$ such that $SS(F) \cap \dot{\pi}^{-1}(U) \subset U \times \gamma$ lies in the image of Waschkies' functor. By construction, Waschkies' map sends (the object in $\mu sh^{pre}_\Lambda(\mathbb{C}^* \cdot p)$ represented by) $F$ to $\Phi_{U, \gamma}(\mathcal{F})$, where $\Phi_{U, \gamma}(-)$ is the microlocal cutoff.  There is always a map $\alpha: \Phi_{U,\gamma}(F) \to F$, and we need to show that the cone is a local system on $U$. Up to shrinking $M$, assume $U=M$. Then $cone(\alpha)$ has no microsupport in $\dot{\pi}^{-1}(M) \setminus (\gamma \times U)$ (because neither $F$ nor $\Phi_{U, \gamma} F$ has microsupport there). But $cone(\alpha)$ also has no microsupport in $U$ because $\Phi_{U, \gamma}$ induces an isomorphism in $\mu sh^{pre}(\gamma \times M)$ \cite[Def.\ 2.3.1(iii)]{waschkies-microperverse}. Hence the microsupport of $cone(\alpha)$ is contained in the zero section, as claimed.
\end{proof}

\begin{remark}
The analogue of \Cref{thm: waschkies-sh-quo-loc} with complex analyticity hypotheses removed
and with $\C^*$ replaced by $\R^+$ is also true (indeed easier: the cutoffs already in \cite{kashiwara-schapira} are good enough
and one does not need \cite{dagnolo-cutoff}); see \cite[Lem 6.7, Prop. 6.9]{nadler-shende}. 
\end{remark}

Now, we compare the two notions of microsheaves $\msh_{T^* M}$ and $\Pmsh_{\P^* M}$. The canonical map
\[\Pmsh_{\P^* M}^\pre(\Omega) \coloneqq \msh^\pre(\pi^{-1}(\Omega))  \rightarrow \msh(\pi^{-1}(\Omega) )\]
defines a map $\Pmsh_{P^* M} \rightarrow \pi_* (\msh |_{T^\circ M})$.

\begin{lemma} \label{lem: P-mush-fully-faithful}
The map $\Pmsh_{\P^* M} \hookrightarrow \pi_* \left( \msh|_{T^\circ M} \right)$ is fully-faithful.
\end{lemma}

\begin{proof}
It's enough to check on stalks. Take $p \in \P^* M$,  the map $(\Pmsh_{\P^* M})_p \rightarrow \left[ \pi_* \left( \mu sh|_{T^\circ M} \right) \right]_p$ is given by
\[ \mu sh^{pre}(\pi^{-1}(p)) \rightarrow \mu sh (\pi^{-1} (p)).\]
As mentioned in \eqref{hom of microsheaves is muhom}, the Hom sheaf on the right-hand side $\cH om_{\mu sh}$ is computed by $\mu hom$. In particular, $\Hom_{\mu sh (\pi^{-1} (p))}$ is computed by $\Gamma(\pi^{-1}(p); \mu hom (-,-)).$ But \cite[Proposition 2.4.4]{waschkies-microperverse} shows that it is also the case for the category on the left-hand side so we are done.
\end{proof}

Since $\mu sh^{pre}(\pi^{-1}(p)) = sh(M)/sh_{T^* M \setminus \pi^{-1}(p)}(M)$ is a quotient, objects in $\mu sh^{pre}(\pi^{-1}(p))$ are presented by sheaves. Thus, the proof of the above lemma characterizes the image as follows:

\begin{corollary} \label{cor: Pmsh-characterization}
The subsheaf $\Pmsh_{\P^* M}$ in $\pi_* (\msh_{T^\circ M})$ is equivalent to the subsheaf stalkwisely on $P^* M$ presented by a sheaf
\[\{\cF \in \pi_* (\mu sh |_{T^\circ M}) | \cF |_p \in \im (sh(M) \rightarrow  \mu sh(\pi^{-1}(p) ), \forall p \in \P^* M \}. \]
\end{corollary}

Similar to the real situation \Cref{thm: real-contact-transform}, there is a complex version of the contact transform.  

\begin{theorem}[{\cite[(11.4.8)]{kashiwara-schapira}}] \label{thm: complex-contact-transform}
Let $\cU \subseteq \P^* M$, and $\cV \subseteq \P^* N$ be open sets, and $\chi: \cU \xrightarrow \cV$ be a complex contactomorphism. Then, for any given $p \in \cU$, shrink $\cU$ if needed, one can assume that there exists a sheaf $K \in sh(M \times N)$ such that the functor $\Phi_K: sh(M) \rightarrow sh(N)$ given by convolving with $K$ induces an equivalence, often referred as contact transformation, 
\[ \Phi_K: \msh^{pre}_{T^\circ M} |_{\widetilde{\cU}} \xrightarrow{\sim } \tilde{\chi}^* \left(\msh^{pre}_{T^\circ N}|_{\widetilde{\cV}}\right) \]
where $\widetilde{\cU}$, $\widetilde{\cV}$ and $\tilde{\chi}$ are the corresponding symplectic lifts.
Consequently, it induces an equivalence $\msh_{T^\circ M}|_{\cU} \xrightarrow{\sim} \tilde{\chi}^* \msh_{T^\circ N} |_{\cV}$ which commutes with the canonical map $\msh^{pre} \rightarrow \msh$.
\end{theorem}

\begin{corollary}\label{cor: Pmsh-contact-transform}
With the notation as above, $\Phi_K$ induces an equivalence ${\pi_{M}}_* \msh_{\widetilde{\cU}} \xrightarrow{\sim} \chi^* {\pi_{N}}_* \msh_{\widetilde{\cV}}$ which restricts to $\Pmsh_{\cU} \xrightarrow{\sim} \chi^* \Pmsh_{\cV}$.    
\end{corollary}

\begin{proof}
The first equivalence is tautological. The second equivalence follows from the characterization of $\Pmsh$ in \Cref{cor: Pmsh-characterization} as locally presentable by sheaves and the fact that contact transformation commutes with the canonical map $\msh^{pre} \rightarrow \msh$ as mentioned in \Cref{thm: complex-contact-transform}.
\end{proof} 

In general, we do not know if the inclusion $\Pmsh_{\P^* M} \hookrightarrow  \pi_* (\msh_{T^\circ M})$ is an equivalence. However, it is the case when we restrict to complex constructible objects.

\begin{proposition}\label{prop: C-c-Pmsh-equal-all} 
The inclusion $\Pmsh_{\P^*M, \C-c} \xhookrightarrow{\sim} \pi_* (\msh_{T^\circ M, \C-c})$  is an equivalence.  
\end{proposition}

\begin{proof}

Let $p \in \P^* M $ and consider $\cF \in [\pi_* (\msh_{T^\circ M, \C-c})]_p = \msh_{T^\circ M, \C-c}(\pi^{-1}(p))$. Since objects in $\msh_{T^\circ M, \C-c}(\pi^{-1}(p))$ are germs of microsheaves near $\pi^{-1}(p)$, we can pick some $\Omega \subseteq \P^* M$ containing $p$ and realize $\cF$ as a microsheaf on $\pi^{-1}(\Omega)$ with $ss(F) \subseteq \pi^{-1}(\Omega)$ being a complex Lagrangian. Denote by $\Lambda \subseteq \Omega$ the corresponding complex Legendrian. 

By the previous \Cref{cor: Pmsh-contact-transform}, being in $\Pmsh$ is invariant under contact transform, so we can apply the Kashiwara--Kawai general position theorem (see \cite[Sec.\ 1.6]{kashiwara-kawai}) and assume that the composition $\Lambda \subseteq \P^* M \xrightarrow{r} M$ is finite to one near $p$. Shrinking $\Omega$ and $M$ if needed, we may assume that $\Lambda$ admits the standard form \cite[(2.5)]{kashiwara-vilonen}: There exists local coordinates $(z,\xi)$ such that $p = d z_n$ and $\Lambda$ is of the form $\P^*_S M$, where $S = \{f = 0\}$ is the zero locus of some holomorphic function $f = z_n^k + g(z)$ for some $k \in \N$ and for some $g(z) \in (z_1, \cdots, z_n)^{k+1}$. 

But this implies that $\cF \in \msh_{\P_S^* M}(\Omega)$ so we can apply \Cref{thm: doubling} and conclude that there exists an $F \in sh(M)$ which projects to $\cF$, which in particular implies that $\cF \in \Pmsh_{P^* M, \C-c}$. 
\end{proof}

\subsection{Canonical microsheaves and microstalks}\label{subsection:canonicalmicrostalkcomplex}
We return to considering an exact symplectic or contact manifold $W$ with corresponding structure morphism 
$W \rightarrow BU$.  Suppose given moreover a lift to a stable quaternionic structure $W \rightarrow BSp$, e.g. arising from an underlying complex symplectic or contact structure as in \Cref{subsection:complex-contact-symplectic}. Per \Cref{def: quarternionic-grading/orientation}, $W$ carries a canonical orientation/grading datum. Per \Cref{thm: guillermou-maslov}, this determines a canonical $R$-Maslov datum, for $R$ a commutative ring. \Cref{def: microsheaves} thus furnishes a sheaf $\mu sh_W$ of $R$-linear categories  on $W$.

More generally, fix a coefficient category $\mathcal{C}$ and an anchored collection of subcategories $\{\mathcal{D}_i\}$. Endow $W$ with the $(\mathcal{C}, \{\mathcal{D}_i\})$ grading induced by the canonical $\mathbb{S}$-grading (see \Cref{def: quarternionic-grading/orientation} and \eqref{equation:chris-big-diagram}), and assume given a $(\mathcal{C}, \{\mathcal{D}_i\})$ orientation $o$. Then \Cref{def: microsheaves} furnishes a sheaf of $\mathcal{C}$-linear categories on $W$ which we denote by $\mu sh_{W, o}$.

\begin{proof}[Proof of \Cref{theorem:intro-1-test}]   Let us first discuss the case $R = \Z$. 
We take the canonical $\Z$-Maslov datum.
Assume $L$ is a complex Lagrangian or Legendrian. Then, the polarization $\phi_L: L \rightarrow BO$ admits a lift to a complex polarization $L \rightarrow BU \rightarrow BO$. To compare it to the fiber polarization $\phi_L$ near L, we consider the secondary Maslov data. By \Cref{thm: guillermou-maslov}  secondary Maslov data are in one-to-one correspondence with secondary orientation/grading. The theorem thus reduces to \Cref{lem: spin}. 

For general $R$, note that \Cref{lem: spin} is itself deduced from \Cref{spin polarization}. Tracing the argument there, the only modification needed is to further compose $B^2 \Z \rightarrow B^2(\Z/2\Z)$ with $B^2(\Z/2\Z) \rightarrow B^2 R^\times$ (the latter map is the twice delooping of $\Z/2\Z = \mathbb{Z}^\times \rightarrow R^\times$). 
\end{proof}

By \Cref{theorem:intro-1-test}, $R$-secondary Maslov data for the canonical Maslov data are given by $R$-spin structures.

\begin{proof}[Proof of \Cref{corollary:test-microsheaf-intro}]
    From \Cref{theorem:intro-1-test} we have a canonical Maslov datum for $W$, giving the sheaf of stable categories $\mu sh_W$ via \Cref{def: microsheaves}. Also from \Cref{theorem:intro-1-test}, an $R$-spin structure $\sigma$ on $L$ determines a secondary Maslov datum which, per \Cref{lem: polarization-local-systems}, determines an equivalence $\msh_{L,\eta} \simeq \msh_{L, \phi_L} \simeq loc_L$. 
\end{proof}

\begin{proof}[Proof of \Cref{corollary:microstalk-intro-test}]
Without loss of generality, we can assume that $X=L$ is a (conical) complex Lagrangian disk. Endow $W$ with the canonical $R$-Maslov datum. According to \Cref{complex legendrian maslov}, there is a canonical choice of secondary grading. So any choice of secondary orientation induces a secondary Maslov datum, and hence, by \Cref{corollary:microstalk-secondarymaslov}, a microstalk functor $
    \omega_p^{-1}: (\mu sh_{L,o})_p \to R-mod$.
By \Cref{corollary:microstalk-first}, $\omega_p^{-1}$ acts unambiguously on the set of objects. 
\end{proof}

We can also generalize \Cref{corollary:microstalk-intro-test} to the ``constrained'' setting (see \Cref{sec: contrained}). Namely, assume that $W$ is an exact complex symplectic manifold and let $L \subset W$ be a (conical, smooth) complex Lagrangian. Fix a coefficient category $\mathcal{C}$ and an anchored collection of subcategories $\{\mathcal{D}_i\}$. Endow $W$ with the canonical $(\mathcal{C}, \{\mathcal{D}_i\})$ grading and assume given a $(\mathcal{C}, \{\mathcal{D}_i\})$ orientation $o$. According to \Cref{complex legendrian maslov}, there is a canonical choice of secondary grading. So any choice of secondary orientation induces a secondary Maslov datum, and hence, by \Cref{corollary:microstalk-secondarymaslov}, a microstalk functor $\omega_p^{-1}: (\mu sh_{L,o})_p \to \mathcal{C}$.
By \Cref{corollary: reduced-ambiguity}, the ambiguity of $\omega_p^{-1}$ is an element of of 
$\Pic(\cC)_{\{\cD_i\}}$. Hence, given one of the $\mathcal{D}_i$, it is meaningful to ask for the image of $\omega_p^{-1}$ to be contained in $\cD_i$. (If $\{\mathcal{D}_i\}$ is the total partition of $\mathcal{C}$ into $1$-object subcategories, $\omega_p^{-1}: (\mu sh_{L,o})_p \to R-mod$ is well-defined as an object, so we recover \Cref{corollary:microstalk-intro-test}.)

\begin{remark} 
    Integrability plays no role in the above arguments. That is, for the purpose of defining the canonical microstalk functor $\omega_p^{-1}$, it would be enough to assume that $W$ is an exact symplectic manifold endowed with a map $X \to BSp$ lifting the classifying map, and that $L$ is a (conical, smooth) Lagrangian endowed with a complex polarization $L \to LGr_\mathbb{C}$.
\end{remark}

Our next task is to compute $\omega_p^{-1}$ on cotangent bundles of complex manifolds, endowed with the Maslov datum induced by the complex polarization. Concretely, we are interested in the case when $L \subset S^*M$ is the conormal to a complex submanifold $N \subset M$ of complex codimension $n- \ell$, for $1 \leq \ell \leq n$. By Darboux, we may assume without loss of generality that $U= J^1\mathbb{R}^{2n}$ and $L$ is the conormal to $\mathbb{R}^{2\ell} \times \{0\}^{2n-2\ell}$.  Let $K \subset J^1(\mathbb{R}^{2n})$ be the zero section, so that $\psi_{\pi/2}(K)=L$, where $\psi_{\pi/2}(-)$ is defined as in \Cref{sec: maslov index}. Henceforth we write $\psi= \psi_{\pi/2}$. 

Assume, as above, that $T^*M$ is endowed with the canonical $(\mathcal{C}, \{\mathcal{D}_i\})$ grading and a fixed $(\mathcal{C}, \{\mathcal{D}_i\})$ orientation $o$.  We denote by $\eta$ the corresponding Maslov datum. We let $\phi_K, \phi_L$ denote the canonical polarizations transverse to $K, L$. We denote by $\psi_*\eta:= \eta \circ \psi^{-1}$ and $\psi_*\phi_K:= d\psi (\phi_K)$ the pushforwards of $\eta$ and $\phi_K$ under $\psi$. Observe that $\psi_*\phi_K:= d\psi (\phi_K)   = \phi_L$, by definition of $\psi$. 
\begin{equation}\label{equation:diagram}
\begin{tikzcd}
\mathcal{C} \ar[r, "\simeq"] \arrow[dd, "{[-2\ell]}"'] & (\mu sh_K)_p \ar[r, "\simeq"] \ar[dd, "(\psi)_*", "\simeq"'] & (\mu sh_{K, \phi_K})_p \ar[dd, bend right=60, "\simeq"'] \ar[d, "\simeq"]
\ar[r, "\omega_p^{-1}" ', "="] & (\mu sh_{K, \phi_K})_p = \mathcal{C} \ar[d, "\simeq"] \\
& & (\mu sh_{L, \phi_L})_p \ar[r, "="] \ar[d] & (\mu sh_{L, \phi_L})_p  = \mathcal{C} \ar[d, dotted, "{[-\ell]}"] \\
\mathcal{C} \ar[r, "\simeq"] & (\mu sh_L)_p \ar[r, "\simeq"] &  (\mu sh_{L,\phi_K})_p \ar[r, "\omega_p^{-1}" ', "\simeq"] & (\mu sh_{L, \phi_L})_p  =\mathcal{C}
\end{tikzcd}
\end{equation}

The dotted arrow means that the rightmost square commutes up to a noncanonical natural transformation (so its effect on objects is unambiguous). 

We first explain the meaning of the arrows in \eqref{equation:diagram}. The leftmost square is defined as in \Cref{lemma:fourier-shift-diagram}. The map $\mu sh_{K, \phi_K} \to \mu sh_{L,\phi_K}$ is induced from the contact transformation $\psi_*: \mu sh_{T^*K, \phi_K} \to \mu sh_{T^*K, \phi_K}$, after stabilizing with $(T^*\mathbb{R}^N, 0_{\mathbb{R}^N}), N \gg 1$. The identification $\mu sh_{K, \phi_K} = \mu sh_{\psi(K), \psi_*\phi_K}$ and the corresponding top-right commutative square are induced by the contactomorphism $\psi$. 

To explain the remaining arrows, recall that $\phi \mapsto (\mu sh_{L, \phi})_p$ forms a local system of categories over $LGr=U/O$, which is classified by $U/O \xrightarrow{BJ} BPic(\mathcal{C})$. The postcomposition $U/O \to BPic(\mathcal{C}) \to B\mathbb{Z}= U(1)$ is precisely $det^2(-)$; hence up to the action of $Pic_0(\mathcal{C})$, the monodromy automorphismm of any loop is precisely shifting by the degree of the loop under $det^2(-)$. 

The canonical microstalk functor 
$\omega_p^{-1}: (\mu sh_{L,\phi_K})_p \to (\mu sh_{L, \phi_L})_p  =\mathcal{C}$ is realized by choosing a homotopy of polarizations $\phi_K \rightsquigarrow \phi_L$ through \emph{complex} Lagrangians, and parallel transporting. Similarly, the arrow $(\mu sh_{L, \phi_L})_p \to (\mu sh_{L, \phi_K})_p$ in induced by the homotopy of polarizations  $\psi_t^{-1}: \psi_* \phi_K= \phi_L \rightsquigarrow \phi_K$. Hence the dotted arrow is induced by the loop of polarizations $\phi_L \rightsquigarrow \phi_K \rightsquigarrow \phi_L$, which defines an automorphism $\mathcal{C}= (\mu sh_{L, \phi_L})_p$. . The first homotopy lies in the kernel of $det^2(-)$, while the image of the path $\psi_t^{-1}: \psi_* \phi_K \rightsquigarrow \phi_K$ under $det^2(-)$ is a loop of degree $-\ell$.

\begin{corollary} \label{cor: complex-microstalk}
    Let $M, \eta$ be as above. Suppose that $L$ is a smooth Lagrangian disk contained in the conormal of a complex  submanifold $N \subset M$ of complex codimension $\ell$, for $1 \leq \ell \leq n$. Let $p \in L$ be a smooth point. Given $A \in \mathcal{C}$, $\omega_p^{-1}(A_N)= A[\ell]$.
\end{corollary}
\begin{proof}
    Without loss of generality, we can assume that $M= \mathbb{R}^{2n}$ and $N = \mathbb{R}^{2\ell} \times \{0\}^{2n-2\ell}$. Let $K=0_{\mathbb{R}^{2n}}$. Then $\omega_p^{-1}(A_K)= A$; by  \eqref{equation:diagram}, $\omega_p^{-1}(A_N[-2\ell]) = \omega_p^{-1}(A_K)[-\ell]$, which proves the claim.
\end{proof}

\subsection{Microsheaves on complex contact manifolds and symplectic manifolds} \label{sec: complex-contact-case}

Let $V$ be a complex contact manifold.  Recall from \Cref{subsection:complex-contact-symplectic} that we have maps
\begin{equation}\label{equation:circle-bundle-diagram}
\begin{tikzcd}
\widetilde{V} \ar[r, "q" '] \ar[rr, bend left=20, "\pi"] & \widetilde{V}/\R_+ \ar[r, "p" '] & V,
\end{tikzcd}
\end{equation}
where $\widetilde{V}$ is an exact complex symplectic manifold with holomorphic $1$-form $\lambda_{\widetilde{V}}$.  For any subset (typically complex Legendrian) $L \subset V$, we similarly write $\widetilde{L} := \pi^{-1}(L)$ and 
$\widetilde{L}/\R_+ := p^{-1}(L)$.

Letting $\hbar \in \C^*$ act by the $\C^*$ principal bundle structure of $\widetilde V$ over $V$, we have $\hbar^* \lambda_{\widetilde V} = \hbar \lambda_{\widetilde V}$ for $\hbar \in \C^*$.  
In particular, multiplication by $\hbar$ descends to a real contactomorphism 
$(\widetilde V / \R_{>0}, \ker \re \lambda) \cong (\widetilde V / \R_{>0}, \ker \re \hbar \lambda)$. 

Per \Cref{def: quarternionic-grading/orientation}, $(\widetilde{V}, \re \lambda)$ carries canonical Maslov data $\eta_{can}$.  We write 
$$\mu sh_{\widetilde{V}}  \coloneqq \mu sh_{(\widetilde{V}, \re \lambda), \eta_{can}}.$$ Here we have, and will henceforth, set $\hbar = 1$.  This is no loss of generality: 
\begin{lemma}
    There is a canonical isomorphism
$\hbar^* \mu sh_{(\widetilde V, \re \hbar \lambda)} \cong \mu sh_{(\widetilde V, \re \lambda)}$.
\end{lemma}
\begin{proof}
    The canonical Maslov datum is pulled back from $V$, so the action of the contactomorphism given by multiplication by $\hbar$ is also canonically trivial on it. 
\end{proof}

More generally, we can consider any Maslov data $\eta$ on $\widetilde{V}$ and the associated sheaf $\msh_{\widetilde{V},\eta}$.

\begin{remark}
The sheaf of categories $\mu sh_{\widetilde V}$
is the pullback of the corresponding object 
on the real contact manifold $\widetilde{V} / \R_{>0}$; we adopt our present formulation solely  to avoid 
the minor cognitive dissonance of passing constantly to a non-complex manifold. 
\end{remark} 

A closed and complex analytic subset $\widetilde{\Lambda} \subset \widetilde{V}$ is
$\R_{> 0}$-invariant iff it is $\C^*$-invariant, 
and hence the preimage of some $\Lambda \subset V$.  We will however also write $\widetilde \Lambda$ for complex and $\R_{> 0}$-invariant (but not necessarily closed or $\C^*$-invariant) subsets of $\widetilde V$.  We say such a subset is Lagrangian if it is contained in the closure of its smooth and Lagrangian locus. 

\begin{definition}
     We write $\mu sh_{\widetilde{V}, \C-c, \eta} \subset \mu sh_{\widetilde{V}, \eta}$ for the sheaf of full subcategories on objects with complex Lagrangian microsupport. 
\end{definition}

By definition, for $\R_{> 0}$-conic subsets $\Omega \subset \widetilde{V}$, the space of sections $\mu sh_{\widetilde{V}, \C-c, o}$ is the union of the categories $\mu sh_{\widetilde{\Lambda}, o}(\Omega)$, where $\widetilde{\Lambda}$ varies over $\R_{> 0}$-conic complex Lagrangian subsets $\widetilde{\Lambda}$ which are closed in $\Omega$.

\begin{remark}
For any fixed subanalytic Lagrangian $\Lambda$, the category $sh_{\Lambda}(M)$ 
is presentable, but $sh_{\C-c}(M)$ is not  (arbitrary sums of constructible sheaves
certainly need not be constructible).  The situation for the microsheaf categories is entirely analogous. 
The distinction is rarely relevant: for example, 
in the present article, while we state theorems for $\mu sh_{\widetilde{V}, \C-c, o}$, 
their proofs quickly reduce to statements about  $\mu sh_{\widetilde{\Lambda}, o}$.  
\end{remark}

\label{section:complex-symplectic}

We now discuss complex symplectic manifolds. Let $W= (W, \lambda)$ be an exact complex symplectic manifold, consider the contact thickening $(W \times \mathbb{C}, \lambda+ dz)$ and the associated symplectization 
\begin{equation}
(W \times \mathbb{C} \times \mathbb{C}^*, w( \lambda + dz)). 
\end{equation}

Let $V_\mathbb{R}$ denote the Liouville vector field for $(W, \lambda)$ and let $I$ denote the almost-complex structure induced by complex multiplication. Integrating the flow of $V_\mathbb{R}$ in the first component and the flow of $V_{S^1}:= I V_\mathbb{R}$ in second, we obtain an action $A: \mathbb{R} \times \mathbb{R} \times W \to W$ which satisfies $A_{(t, \theta)}^* \lambda = e^{t+ 2\pi i\theta}\lambda$. 

If $U \subset \mathbb{C}^*$ is contractible, then we can define an action $U \times W \to W$ by lifting $U$ to $\mathbb{R} \times \mathbb{R}$ by the covering map $(t, \theta) \mapsto e^{t+ 2\pi i \theta}$.  In many situations, the flow in the $I Z$ direction factors through $k\mathbb{Z} \subset \mathbb{R}, k \geq 1$; in this case, we can lift the action via the covering map $(t, \theta) \mapsto e^{k(t+ 2\pi i \theta)}$ to define a weight $k$ action of $\mathbb{C}^*$. 

Let us now suppose that $U \subset \mathbb{C}^*$ is a ball containing $1$, and consider the induced weight-$1$ action $U \times X \to X, (w, x) \mapsto A_w(x)$, $A_w^*(\lambda)= w\lambda$. We record the following convenient change of variables:
\begin{align}\label{equation:change-variables}
(W \times \mathbb{C} \times U, \lambda + w dz) &\xrightarrow{\simeq} (W \times \mathbb{C} \times U, w (\lambda +  dz))\\
(x, z, w) &\mapsto (A_{w^{-1}}(x), z, w)\nonumber
\end{align}

Given an exact complex-symplectic manifold $(W, \lambda)$, consider the contactization $(W \times \mathbb{C}, \lambda + dz)$. This is a complex-contact manifold, and so we can consider $\mu sh_{W \times \{0\}} \subset \mu sh_{W \times \mathbb{C}}$, which is the full subcategory on objects microsupported on $W \times \{0\}$. We also write $\mu sh_{W \times \{0\}, \mathbb{C}-c}$ for the full subcategory of $  \mu sh_{W \times \{0\}}$ on objects with complex constructible support. 

We now have two sheaves of categories on $W$, namely $\mu sh_W \coloneqq \msh_{(W,\re \lambda)}$ from \Cref{def: microsheaves} and $\mu sh_{W \times \{0\}}$.  They are not the same: for example, if $W$ is a point, then $\mu sh_W= C$ while $ \mu sh_{W \times \{0\}}$ is the category of local systems on $\mathbb{C}^*$.  

\begin{theorem} \label{weight-1 stuff} 
Suppose that the Liouville vector field of $(W, \lambda)$ integrates to a weight-1 $\C^*$-action. Let $\gamma_\mathbb{C}: W \to W$ be the set-theoretic identity, where the source is endowed with the Euclidean topology and the target with the $\mathbb{C}^*$-invariant topology.
Then there is a natural $\Z = \Omega S^1$-linear structure on $(\gamma_\mathbb{C})_* \mu sh_{W \times \{0\}}$ and an equivalence 
\begin{equation} \label{pmush invts to mush}
    (\gamma_\mathbb{C})_*\left(  \mu sh_{W \times \{0\}} \otimes_{\Omega S^1} \bullet \right) \simeq (\gamma_\mathbb{C})_*\mu sh_{W}.
\end{equation}
Furthermore, this equivalence respects complex constructibility. 
\end{theorem}
\begin{proof}
Up to replacing $W$ by a $\mathbb{C}^*$-invariant open, it is enough to exhibit the $\Omega S^1$-linear structure on $ \mu sh_{W \times \{0\} }(W \times \{0\})$, and to prove the equivalence \eqref{pmush invts to mush} on global sections. 

By our assumption on Liouville vector field, the change of variables \eqref{equation:change-variables} is global:
\begin{align}\label{equation:change-variables2}
(\widetilde{W \times \mathbb{C}} = W \times \mathbb{C} \times \mathbb{C}^*, w (\lambda +  dz)) &\xrightarrow{\simeq} (W \times \mathbb{C} \times \mathbb{C}^*, \lambda + w dz).  
\end{align}
The Liouville structure on the right hand side is a product; we have the  K\"unneth isomorphism:
\begin{equation}\label{equation:kun-iso}
\mu sh_{W \times 0 \times \C^*} = \mu sh_{W} \boxtimes \mu sh_{\C^* \subset T^* \C^*}.
\end{equation}

By definition
\[\mu sh_{W \times \{0\}}(W \times \mathbb{C})= \left(\pi_* \mu sh_{\pi^{-1}(W \times \{0\})} \right)(\widetilde{W \times \mathbb{C}}),\] 
where $\mu sh_{\pi^{-1}(W \times \{0\})}$ denotes the sheaf of full subcategories of $\mu sh_{\widetilde{W \times \mathbb{C}}}$ on objects whose support is contained in $\pi^{-1}(W \times \{0\})$; see \eqref{equation:circle-bundle-diagram}. But by \eqref{equation:change-variables2} and \eqref{equation:kun-iso}, we have \[\left(\pi_* \mu sh_{\pi^{-1}(W \times \{0\})}\right)(\widetilde{W \times \mathbb{C}})= \mu sh_{W}(W) \otimes loc(\mathbb{C}^*).\] 
\end{proof}

\begin{remark}\label{remark:diagonal-vs-factor}
Let $\mathbb{C}^* \times W \to W, (\theta, z)\mapsto \theta \cdot z$ be the weight-$1$ $\mathbb{C}^*$ action on $W$. Then $(\theta; x, z) \mapsto (\theta \cdot x, \theta z)$ defines a $\mathbb{C}^*$ action on $W \times \mathbb{C}$ by contactomorphism, which fixes $W \times \{0\}$ set-wise. \eqref{pmush invts to mush} amounts to taking invariants of this action; see \cite[Sec.\ 6]{McBreen-Shende-Zhou}. 
\end{remark}

\section{The perverse t-structure} \label{sec: perverse} 
\subsection{t-structures} 
The notion of a \emph{$t$-structure} on a triangulated category was introduced in \cite{bbd}.  
We recall the definition and some basic properties.

\begin{definition}\label{definition:t-structure} \label{definition t}
Let $\mathcal{T}$ be a triangulated category. 
A pair of subcategories $\mathcal{T}^{\le 0}, \mathcal{T}^{\ge 0}$ determine a \emph{t-structure} if the following conditions are satisfied:
\begin{itemize}
\item[(i)] 
For any $K' \in T^{\le 0}$ and $K'' \in T^{\ge 0}$, 
we have $\Hom(K', K''[-1]) = 0$. 
\item[(ii)] If $K' \in \mathcal{T}^{\le 0}$ then $K'[1] \in \mathcal{T}^{\le 0}$; similarly if $K'' \in \cT^{\ge 0}$ then $K''[-1] \in \mathcal{T}^{\ge 0}$. 
\item[(iii)] Given $K \in \mathcal{T}$, there exist $K' \in \mathcal{T}^{\le 0}$ and $K'' \in \mathcal{T}^{\ge 0}$, and 
a distinguished triangle $$K' \to K \to K''[-1] \xrightarrow{[1]}$$
\end{itemize}
We write $\mathcal{T}^{\ge n} := \mathcal{T}^{\ge 0}[-n]$ and $\mathcal{T}^{\le n} := \mathcal{T}^{\le 0}[-n]$.  The {\em heart} of the t-structure is $\mathcal{T}^{\heartsuit} := \mathcal{T}^{\leq 0} \cap \mathcal{T}^{\geq 0} \subset T$. 
\end{definition}

It is shown that there are {\em truncation functors}
$\tau^{\le n} : \cT \to \cT^{\le n}$ and $\tau^{\ge n} : \cT \to \cT^{\ge n}$ which are 
right and left adjoint to the inclusions of the corresponding subcategories.  
The truncation functors commute in an appropriate sense
(e.g. when composed with the inclusions so as to define endomorphisms of $\cT$).  
Then $H^0 := \tau^{\le 0} \tau^{\ge 0} = \tau^{\ge 0} \tau^{\le 0}$ defines a map 
$\cT \to \cT^{\heartsuit}$, and one writes $H^n: \cT \to \cT^{\heartsuit}$ for the appropriate
composition with the shift functor. 
Finally,  $\cT^\heartsuit$ is an abelian category,
closed under extensions \cite[Thm.\ 1.3.6]{bbd}.

The prototypical example is when $\cT$ is a derived category of chain complexes, 
 $\cT^{\le 0}$ (resp. $\cT^{\ge 0}$) consists of the complexes whose cohomology is 
concentrated in degrees $\le 0$ (resp. $\ge 0$).

Let us recall a result about when $t$-structures pass to quotient categories. 

\begin{lemma}[Lem.\ 3.3 in \cite{chuang2017perverse}]\label{lemma:induced-t-structure}
Let $\mathcal{T}^{\leq 0}, \mathcal{T}^{\geq 0}$ determine a $t$-structure on $\cT$. 
Let $\cI \subset \cT$ be a triangulated subcategory, closed under taking direct summands
(``thick subcategory'') and let $Q: \cT \to \cT/\cI$ be the Verdier quotient.  Then: 
\begin{enumerate}
\item 
$\mathcal{T}^{\leq 0} \cap \mathcal{I}, \mathcal{T}^{\geq 0} \cap \mathcal{I}$ determine
a $t$-structure if and only if $\tau_{\leq 0} \mathcal{I} \subset \mathcal{I}$ 
\item if the equivalent assertions of (1) hold, 
 $Q(\mathcal{T}^{\leq 0}), Q(\mathcal{T}^{\geq 0} )$ determine a $t$-structure
 if and only if 
$\mathcal{I} \cap \cT^\heartsuit \subset \cT^\heartsuit$ is a ``Serre subcategory'' (meaning it is closed under extensions, quotients and sub-objects). 
\end{enumerate}
\end{lemma}

The notion of $t$-structure is imported to the setting of stable categories in \cite[Sec.\ 1.2]{lurie-ha}: 
By definition, a $t$-structure on a stable category is a $t$-structure on its homotopy category,
which canonically carries the structure of a triangulated category.  It is shown that the 
various properties of $t$-structures lift to the stable setting, in particular, the existence
of truncation functors, and the fact that the full subcategory on objects in the heart is abelian.\footnote{Let us avoid a possible source of confusion.  One might think that, insofar as stable categories generalize dg categories, the heart 
could be expected to have, in its hom spaces, whatever corresponds to the positive ext groups.  
This depends on whether or not the stable category is viewed as a usual $\infty$-category, or as an $\infty$-category enriched in spectra.
Indeed, the positive ext groups (in cohomological grading conventions) correspond to negative homotopy groups, so are only manifest after the 
(canonical) enrichment in spectra.  Here however the statement about the heart should be understood
in terms of the not enriched $\infty$-category.}   
For consistency with \cite{bbd} (and in contrast to \cite{lurie-ha}), we adopt cohomological conventions and write $H^i$ instead of $\pi_{-i}$. 

\begin{remark} \label{t compact objects}
When $\cC$ is a presentable stable category, then if either $\cC^{\ge 0}$ or $\cC^{\le 0}$ is presentable, 
then so is the other, and all truncation functors are colimit preserving \cite[1.4.4.13]{lurie-ha}.  In this case, 
the subcategory of compact objects $\cC^c$ is stable under the truncation functors and inherits
a $t$-structure.  Indeed, $\tau^{\ge 0}$ is left adjoint to the corresponding inclusion, assumed colimit preserving,
hence $\tau^{\ge 0}$ preserves compact objects.  Taking cones, so does $\tau^{\le 0}$.  
\end{remark} 

We will study sheaves of $t$-structures on sheaves of categories. 

\begin{definition} \label{sheaf t structure} 
Let $M$ be a topological space and $\mathcal{F}$ a sheaf of stable categories on $M$.  
We say a pair of sheaves  of full subcategories  $\mathcal{F}^{\le 0}$ and $\mathcal{F}^{\ge 0}$  
define a $t$-structure on $\mathcal{F}$ 
if $\mathcal{F}^{\le 0}(U)$ and $\mathcal{F}^{\ge 0}(U)$  define a $t$-structure on 
$\mathcal{F}(U)$ for all $U$. 
\end{definition} 

\begin{lemma}
    \label{lem: t-structure-local}
The property that  $\mathcal{F}^{\le 0}$  and  $\mathcal{F}^{\ge 0}$ define a $t$-structure may be checked
on sections on any base of open sets.
\end{lemma}
\begin{proof}  Indeed, regarding condition (i) and (ii) of \Cref{definition:t-structure}, it is immediate from
the sheaf condition 
that vanishing of Homs and containment of subcategories can be checked locally.  

Regarding (iii), the key point is that 
for any candidate $t$-structure satisfying (i) and (ii), 
the space of fiber sequences $K' \to K \to K''$ as requested in (iii) is either empty or contractible.  
Indeed, first recall that using property (ii) to apply (i) to shifts, we find the following strengthening of (i): for any 
$K' \in T^{\le 0}$ and $K'' \in T^{\ge 0}$, the negative exts, aka positive homotopy groups of the hom {\em space} $\Hom(K', K''[-1])$,
must vanish.  Now given any $K' \to K \to K''[-1]$ and
 $L' \to K \to L''[-1]$ both satisfying (iii), we obtain a canonical null-homotopy of the composition $K' \to K \to L''[-1]$ hence lift of
 $K' \to K$ to $K' \to L'$, etc.  
 
 Having learned this contractibility, if (iii) holds locally, then we can canonically glue the local exact triangles to obtain (iii) globally. 
\end{proof} 

Since pullbacks commute with limits, 
$\mathcal{F}^\heartsuit := \mathcal{F}^{\leq 0} \cap \mathcal{F}^{\geq 0}$
defines a sheaf of $(\infty,1)$-categories.  As hearts of $t$-structures,
these categories are abelian, in particular, 1-categories.

\begin{remark}
In the classical literature, the sheaf condition for sheaves of 1-categories (such sheaves are sometimes called stacks) is formulated as a limit of 1-categories taken in the $(2, 1)$-category of ordinary categories.  In this formulation, 
the compatibility condition on triple overlaps is strict.  

By contrast, the notion of  $\infty$-categorical sheaf of $(\infty, 1)$-categories requires that for such a sheaf $\cC$ and covers
$U = \bigcup U_i$, one has
$$\cC(U) = \lim \left( \prod_{i \in I} \cC(U_i) \rightrightarrows \prod_{i, j \in I} \cC(U_i \cap U_j)
 \mathrel{\substack{\textstyle\rightarrow\\[-0.6ex]
                      \textstyle\rightarrow \\[-0.6ex]
                      \textstyle\rightarrow}}
\prod_{i, j, k \in I} \cC(U_{i} \cap U_j \cap U_k) \cdots \right).
$$

These notions are equivalent for $1$-categories: one passes from the $\infty$-categorical notion to the $1$-categorical notion by truncation, and for the reverse direction one need only note that  $1$-categories are $1$-truncated objects 
of $(\infty, 1)$-categories, and the inclusion of $k$-truncated objects is limit-preserving  \cite[Proposition 5.5.6.5]{lurie-htt}.  
\end{remark}

\subsection{The perverse t-structure on constructible sheaves}\label{subsection:perverse-t-structure-summary}
We now review from \cite[Sec.\ 10.3]{kashiwara-schapira} the microlocal description of the perverse $t$-structure on constructible sheaves. 

Let $M$ be a complex manifold.  For a  Lagrangian subset $\Lambda \subset T^*M$, we write $\Lambda^{\circ}$ for the locus of smooth points of $\Lambda$ where the map $\Lambda \to M$ has locally constant rank. Fix a $t$-structure $\cC^{\le 0}, \cC^{\ge 0}$ on our coefficient category $\cC$, with corresponding truncation functors $\tau^{\leq 0}, \tau^{\geq 0}$. 
It is proved in \cite[Theorem 10.3.12]{kashiwara-schapira} that the following prescription characterizes the perverse $t$-structure on $sh(M)_{\C-c}$.\footnote{The results stated in \cite{kashiwara-schapira} are for $\mathcal{C}$ the bounded derived category of modules over a ring and $t$ the standard $t$-structure.  However, the arguments given there (or in \cite{bbd}) for the existence of the perverse $t$-structure depend only on the general properties  of the six functor formalism, and the comparison between Definition \ref{definition:perverse-t-structure} and the usual stalk/costalk wise definition of the perverse $t$-structure depends only on standard properties of microsupports.}  

\begin{definition}[{\cite[(10.3.7) and Definition 10.3.7]{kashiwara-schapira}}] \label{definition:perverse-t-structure} 
Let ${}^\mu sh(M)_{\C-c}^{\leq 0}$ (resp.\  ${}^\mu sh(M)_{\C-c}^{\geq 0}$) be the full subcategory of $sh(M)_{\C-c}$ on objects $F$ with the property that, for every $p \in ss(F)^{\circ}$ such that $\pi: SS(F) \to M$ has constant rank on a neighborhood of $p$, there exists a submanifold $N$ and $L \in \mathcal{C}$ such that $F \simeq L_Y[dim(N)] \in (\mu sh_{T^*M})_p$ and $\tau^{\geq 1}F \simeq 0$ (resp.\ $\tau^{\leq -1}F \simeq 0$).  
\end{definition}

\Cref{definition:perverse-t-structure} can be equivalently expressed using the microstalk functor of \Cref{corollary:microstalk-intro-test}.

\begin{definition}\label{definition:our-perverse-t-structure} 
Consider the following full subcategories of $sh(M)_{\C-c}$. 
$$ {}^\mu sh(M)_{\C-c}^{\leq 0} := \{F \in  sh(M)_{\C-c}\, | \,  p \in ss(F)^{\circ} \implies \omega_p^{-1} F[-n]  \in \cC^{\le 0} \} $$
$$ {}^\mu sh(M)_{\C-c}^{\geq 0} := \{F \in  sh(M)_{\C-c}\, | \,  p \in ss(F)^{\circ} \implies \omega_p^{-1} F[-n]  \in \cC^{\ge 0}\} $$
\end{definition}

\begin{proposition} \label{prop: t-structure}
\Cref{definition:our-perverse-t-structure} and \Cref{definition:perverse-t-structure} agree.
\end{proposition}
\begin{proof} 
Follows immediately from \Cref{cor: complex-microstalk}. 
\end{proof}

\begin{lemma}\label{lemma:achar-facts} 
If $\Lambda \subset T^*M$ is (possibly singular) subanalytic complex Lagrangian and $\Omega \subseteq \P^* M$ is an open set, then \Cref{definition:perverse-t-structure} induces a $t$-structure on 
\[ sh_{\Lambda \cup \pi^{-1}(\Omega^c), \C-c}(M) = \{F \in sh_{\C-c}(M)| ss(F) \cap \pi^{-1}(\Omega) \subseteq \Lambda\}.\] 

Moreover, $sh_{\Lambda \cup \pi^{-1}(\Omega^c), \C-c}(M)^\heartsuit$ is closed under extensions inside $sh(M)$. That is, if $F' \to F \to F''$ is a fiber sequence in $sh(M)$ such that $F', F''$ are in $sh_{\Lambda \cup \pi^{-1}(\Omega^c), \C-c}(M)^\heartsuit$, then $F$ is also in $sh_{\Lambda \cup \pi^{-1}(\Omega^c), \C-c}(M)^\heartsuit$. In particular, $0 \to F' \to F \to F'' \to 0$ forms a short exact sequence in $ sh_{\Lambda \cup \pi^{-1}(\Omega^c), \C-c}(M)^\heartsuit$.
\end{lemma}

\begin{proof}
By (1) of \Cref{lemma:induced-t-structure}, we only need to check that $\tau_{\leq k} F$ is contained in the subcategory if $F$ is.   Suppose not; then $\left( ss(\tau_{\leq k} F) \cap \pi^{-1}(\Omega) \right) \setminus \Lambda$ is non-empty. This means that there exists some $p \in \Omega$ such that any neighborhood $U$ of $\mathbb{C}^*\cdot p$ would have nontrivial intersection with $ss(\tau_{\leq k} F)  \setminus \Lambda$. Since $\tau_{\leq k} F$ is constructible, $(ss(\tau_{\leq k} F) \setminus \Lambda) \cap U$ must have a smooth Legendrian point $q$.  But then the microstalk of $\tau_{\leq k} F$ at this point $q$ is the truncation of the microstalk of $F$, which is zero. A contradiction.

Now suppose given $F', F'' \in sh_{\Lambda \cup \pi^{-1}(\Omega^c), \C-c}(M)^{\heartsuit}$ and some extension $F$, i.e., an object $F \in sh(M)$ fitting in a fiber sequence $F' \to F \to F''$.  The microlocal triangle inequality \cite[Proposition 5.1.3]{kashiwara-schapira} implies that $F \in sh_{\Lambda \cup \pi^{-1}(\Omega^c), \C-c}(M)$ so it's sufficient to check that it is in the heart.  As the microstalk functors are $t$-exact by construction, we get a corresponding extension of microstalks inside $\cC$. But $\cC^\heartsuit$ is closed under extensions so all microstalks of $F$ lie in $\cC^\heartsuit$, and thus $F \in sh_{\Lambda}(M)^\heartsuit$. 
The same argument also shows that $sh_\Lambda(M)^\heartsuit$ is closed under quotients and subobjects.
\end{proof}

\begin{theorem} \cite{waschkies-microperverse} \label{thm: microperverse t structure in cotangent} 
The microlocal perverse t-structure in \Cref{definition:our-perverse-t-structure} induces a perverse $t$-structure on $\pi_*(\msh_{T^\circ M, \C-c})$. Furthermore, for $\cF, \cG \in [\pi_*(\msh_{T^\circ M, \C-c}) ]^\heartsuit$, the hom-sheaf 
\[\cH om_{\pi_*(\msh_{T^\circ M, \C-c})}(\cG,\cF)[ \dim M]\] 
is a perverse sheaf on the $T^\circ M$.
\end{theorem}

\begin{proof}
By \Cref{prop: C-c-Pmsh-equal-all}, the inclusion $\Pmsh_{\P^* M, \C-c} \xhookrightarrow{\sim} \pi_*(\msh_{T^\circ M, \C-c})$ is an equivalence so we can work with $\Pmsh_{\P^* M, \C-c}$. Any $\cF \in \Pmsh_{P^* M, \C-c}(\Omega)$ tautologically belongs to $\Pmsh_{\Lambda}(\Omega)$ where $\Lambda \coloneqq \supp(\cF)$, so we could fix a support condition $\Lambda$. By \Cref{lem: t-structure-local}, it's enough to check it on an open cover. But since $\Pmsh_\Lambda$ is constructible, there exists a cover $\cU$ such that each $\Omega$ satisfies 
$\Pmsh_\Lambda(\Omega) = (\Pmsh_\Lambda)_p$ for some $p \in \Omega$, so we can check on stalks. By \Cref{cor: Pmsh-contact-transform} and \Cref{prop: t-structure}, $(\Pmsh_\Lambda)_p$ and the notion of microstalks are invariant under contact transform, so we can assume $\Lambda$ is in general position. In this case, by \Cref{thm: waschkies-sh-quo-loc}, $(\Pmsh_\Lambda)_p \hookrightarrow (sh_\Lambda)_m/loc_m$ has image given by the category in the previous \Cref{lemma:achar-facts}, which we have shown to have a t-structure, and it clearly induces a $t$-structure on the quotient by applying \Cref{lemma:induced-t-structure} to $loc$.

For the statement regarding the hom sheaf, let $\cF, \cG \in (\pi_* \msh_{T^\circ M,\C-c})^\heartsuit$.  Since perversity can be checked locally, we may pick some sheaves $F$ and $G$ representing $\cF$ and $\cG$. In this case, by (\ref{hom of microsheaves is muhom}), we have the identification
\[\cH om_{\pi_* \msh_{T^\circ M,\C-c}}(\cG,\cF) = \cH om_{\pi_* \msh_{T^\circ M,\C-c}}(G,F) = \mu hom(G,F).\]
Then, \cite[Corollary 10.3.20]{kashiwara-schapira} implies that $\cH om_{\pi_* \msh_{T^\circ M,\C-c}}(\cG,\cF)[\dim M] = \mu hom(G,F)[\dim M]$ is perverse.
\end{proof}

\subsection{Perverse microsheaves on complex contact and symplectic manifolds}

We now define notion of perverse t-structure for microsheaves on complex symplectic and contact manifolds for the canonical microsheaves over $\cC = R -mod$, for a discrete commutative ring $R$. We postpone the general discussion later in \Cref{def: exotic-microperverse-t-structure} for conceptual clarity.

Recall by \Cref{corollary:microstalk-intro-test}, for any microsheaf $\cF$ on $V$, there is a well-defined object $\omega_p^{-1}(F) \in \cC$.

\begin{definition} \label{def: canonical-t-structure}
Let $V$ be a contact manifold of complex dimension $2n-1$. We define the pair of subcategories $( (\mu sh_{\widetilde V, \C-c})^{\le 0}, (\mu sh_{\widetilde V, \C-c})^{\ge 0})$ by constraining the microstalks:
\[(\mu sh_{\widetilde V, \C-c})^{\le 0} := \{\cF \in  \msh_{V,\C-c}\, | \,  p \in \supp(\cF)^{\circ} \implies \omega_p^{-1} \cF[-n]  \in \cC^{\le 0} \}\] 
\[(\mu sh_{\widetilde V, \C-c})^{\ge 0} := \{\cF \in  \msh_{V,\C-c}\, | \,  p \in \supp(\cF)^{\circ} \implies \omega_p^{-1} \cF[-n]  \in \cC^{\ge 0} \}\] 
We define similarly the corresponding notions for objects supported in some fixed
(singular) Lagrangian, and 
define as always the corresponding notions on $V$ by pushforward. \footnote{The $[n]$ is just a convention, set to match the usual conventions for perverse sheaves.}
\end{definition}

\begin{theorem}\label{theorem:main-comparison-local}
$ (  \msh_{V, \C-c})^{\le 0}$, $(  \msh_{V, \C-c})^{\ge 0}$ determine
a $t$-structure on $  \msh_{V, \C-c}$. 
In particular, 
$( \msh_{ V, \C-c})^\heartsuit$ is a sheaf of abelian categories.

Furthermore, for $\cF, \cG \in \msh_{V,\C-c}^\heartsuit$, the sheaf of morphisms 
$$\cH om_{\msh_V}(\cG,\cF)[\frac{1}{2} \dim \widetilde{V}]$$ 
is a perverse sheaf on the symplectization $\widetilde{V}$.  
\end{theorem}
\begin{proof}
Per \Cref{lem: t-structure-local}, a pair of subcategories $(\cC^\leq,\cC^\geq)$ being a t-structure can be checked on open covers. Similarly, being a perverse sheaf if also a local condition. Thus, we reduce to \Cref{thm: microperverse t structure in cotangent} by taking Darboux charts. 
\end{proof} 

We deduce \Cref{theorem:main symplectic intro} from the introduction. 

\begin{proof}[Proof of \Cref{theorem:main symplectic intro}]
By \Cref{weight-1 stuff}, there is an equivalence $(({\gamma_\C})_*\mathbb{P} \mu sh_{W \times \{0\}})^{\mathbb{C}^*} \simeq ({\gamma_\C})_*\mu sh_{W}.$
The $\mathbb{C}^*$-action manifestly preserves the subspaces $\mu sh_{W, \C-c}^{\le 0}$ and $\mu sh_{W, \C-c}^{\ge 0}$, so the $t$-structure passes to the invariants (taking hearts is a pullback ($\mathcal{F}^\heartsuit := \mathcal{F}^{\leq 0} \cap \mathcal{F}^{\geq 0}$) hence a limit, and hence commutes with taking $G$-invariants which is also a limit).

Now, we consider the perversity of sheaf Hom. We recall that objects here admits two different interpretations: as microsheaves on $\widetilde{{W \times \C}}$ and as microsheaves on the underlying real $W$ since \Cref{weight-1 stuff} is needed to define $\mu sh_{W, \C-c}^{\heartsuit}$. Denote by $\widetilde{\cF}$, $\widetilde{\cG}$ for the former and $\cF$, $\cG$ for the latter, then we see from the contact case that $\mu hom(\widetilde{\cG},\widetilde{\cF})[\frac{1}{2}(\dim W) + 1)]$ is a perverse sheaf on $\widetilde{W \times \C}$. However, to descend them onto $W$, we recall that such objects are assumed to be locally constant along the $\C^*$-action and we take $\C^*$-invariant to quotient the extra direction. This process drops the dimension by $1$ and we thus conclude that $\mu hom(\cG,\cF)[\frac{1}{2} \dim W]$ is perverse.
\end{proof}

Now, let $\cC$ be any symmetric monoidal category and $(\cC^{\leq 0}, \cC^{\geq 0})$ be a $t$-structure on $\cC$. We explain how one can generalize \Cref{def: canonical-t-structure} to allow more flexibility on Maslov data. For this purpose, we recall the notion of constrained Maslov data from \Cref{sec: contrained}: A collection of subcategories (of $\cC$) $\{ \cD_i \}$ is said to be anchored if the submonoid $\Pic(\cC)_{\{\cD_i\}}$ fixing each $\cD_i$ is a subgroup.

\begin{lemma} \label{lem: t-structure-anchored}
Let $(\cC^{\leq 0}, \cC^{\geq 0})$ be a $t$-structure on $\cC$, then the collection $\{\cC^{\leq 0}, \cC^{\geq 0}\}$ is anchored. 
\end{lemma}

\begin{proof}
It is sufficient to show that, for $x \in \Pic(\cC)$, tensoring $x \otimes (-)$ fixes $\cC^{\leq 0}$ if and only if tensoring its inverse $x^{-1} \otimes (-)$ fixes $\cC^{\geq 0}$. We show the ``if" direction, since the argument is symmetric. As remarked in \cite[Remark 1.2.1.3]{lurie-ha}, an object $a \in \cC$ is in $\cC^{\geq 1}$ if $\Hom(b,a) = 0$ for all $b \in \cC^{\leq 0}$. Thus we consider such objects and compute that $\Hom(b, x^{-1} \otimes a) = \Hom(x \otimes b, a) = 0$. This implies that $x^{-1} \otimes (-)$ fixes $\cC^{\geq 1}$ but, since $[n]$ commutes with tensor, it fixes $\cC^{\geq n}$ for all $n \in \Z$.
\end{proof}

Since the collection of subcategories $\{\cC^{\leq 0}, \cC^{\geq 0}\}$ is anchored, there is a notion of constrained gradings and orientations defined in \Cref{def: secondary-orientation}. As explained by \Cref{equation:chris-big-diagram}, the canonical grading from the complexes structure induces a $(\cC, \{\cC^{\leq 0},\cC^{\geq 0} \})$-grading. 

\begin{definition}
Let $o$ be a  $(\cC, \{\cC^{\leq 0},\cC^{\geq 0} \})$-orientation, i.e., a lifting of the induced grading to a $\cC$-Maslov data. We denote by $\msh_{\widetilde{V},o}$ the associated sheaf of microsheaves on $\widetilde{V}$ and similar notations for the subsheaves with support conditions. 
\end{definition}
Let $L$ be a complex Legendrian. Then \Cref{complex legendrian maslov} allows us to choose constrained secondary Maslov data, in the sense defined above \Cref{corollary: reduced-ambiguity}, and the cited Corollary implies that, for any constructible microsheaf $\cF \in \msh_{\widetilde{V},\C-c,o}$ and any smooth point $p \in \supp(\cF)^\circ$, whether the microstalk $\omega_p^{-1}\cF$ is in $\cC^{\leq 0}$ or $\cC^{\geq 0}$ (and hence all their shiftings) is a well-defined notion. Thus, we have the following generalization of \Cref{def: canonical-t-structure}:

\begin{definition} \label{def: exotic-microperverse-t-structure} 
Fix a $t$-structure on the coefficient category $\cC$. Let $V$ be a contact manifold of complex dimension $2n-1$. 
Fix any $(\cC, \{\cC^{\leq 0}, \cC^{\geq 0}\})$-orientation  data $o$. We define subcategories of $\mu sh_{\widetilde V, \C-c, o}$ by constraining the microstalks
to respect the $t$-structure of $\cC$:
\[(\mu sh_{\widetilde V, \C-c, o})^{\le 0} := \{\cF \in  \msh_{V,\C-c,o}\, | \,  p \in \supp(\cF)^{\circ} \implies \omega_p^{-1} \cF[-n]  \in \cC^{\le 0} \}\] 

\[(\mu sh_{\widetilde V, \C-c, o})^{\ge 0} := \{\cF \in  \msh_{V,\C-c, o}\, | \,  p \in \supp(\cF)^{\circ} \implies \omega_p^{-1} \cF[-n]  \in \cC^{\ge 0} \}\]
We define similarly the corresponding notions for objects supported in some fixed
(singular) Lagrangian, and 
define as always the corresponding notions on $V$ by pushforward.  
\end{definition} 

As perversity can be checked locally, the exact same argument of \Cref{theorem:main-comparison-local} and \Cref{theorem:main symplectic intro} implies the following theorem.

\begin{theorem}\label{theorem:main-comparison}
Fix a $t$-structure on the coefficient category $\cC$. Let $V$ be a contact manifold of complex dimension $2n-1$. Or, similarly, let $W$ be an exact complex symplectic manifold of complex dimension $2n$ with a $\C^*$-action of weight $1$. 

Given any $(\cC, \{\cC^{\leq 0}, \cC^{\geq 0}\})$-orientation  data $o$ on $V$ (resp. $W$), the pairs 
\[\left( (\msh_{V, \C-c, o }\right)^{\le 0}, (  \msh_{V, \C-c, o})^{\ge 0}  ) \left(\text{resp.} \ \left( \left((\gamma_\mathbb{C})_* \mu sh_{W, \C-c}\right)^{\geq 0}\right), \left((\gamma_\mathbb{C})_* \mu sh_{W, \C-c}\right)^{\leq 0} \right) \] determine
a $t$-structure on $\msh_{V, \C-c, o}$ (resp. $(\gamma_\mathbb{C})_* \mu sh_{W, \C-c}$). Furthermore, for $\cF, \cG \in \msh_{V,\C-c,o}^\heartsuit$ $\left(resp. \left( (\gamma_\mathbb{C})_* \mu sh_{W, \C-c})^\heartsuit \right) \right)$, the sheaf of morphisms 
\[\cH om_{\msh_V}(\cG,\cF)[\frac{1}{2} \dim \widetilde{V}] \left( \text{resp.} \ \cH om_{(\gamma_\C)_*\msh_W}(\cG,\cF)[\frac{1}{2} \dim W] \right) \]
is a perverse sheaf on the symplectization $\widetilde{V}$ (resp. $W$).  
\end{theorem}

\begin{remark}
A $t$-structure is said to be nondegenerate if $\bigcap \cC^{\le 0} = 0 = \bigcap \cC^{\ge 0}$. 
By co-isotropicity of microsupport, the vanishing of all microstalks implies the vanishing of an object. 
We conclude that if the $t$-structure on $\cC$ is non-degenerate, then the $t$-structure on $\mu sh_{V, \mathbb{C}-c, o}$ is also non-degenerate. 
\end{remark}

\begin{remark}
Choose a stable (not necessarily presentable) subcategory $\cD \subset \cC$ to which the $t$-structure
restricts.  Require $(\cC, \{ \cC^{\le 0}, \cC^{\ge 0}, \cD\})$-orientation data.  Then it is evident from the definitions
that the truncation functors preserve, hence define a $t$-structure on, the subcategory of objects
with all microstalks in $\cD$, characterized by the same formulas save only with e.g. $\cC^{\le 0}$ replaced
by $\cC^{\le 0} \cap \cD$.  For instance, we can take various bounded categories e.g.
 $\cD = \cC^+ = \bigcup \cC^{\ge n}$, or $\cD = \cC^- = \bigcap \cC^{\le n}$, or $\cD = \cC^b = \cC^+ \cap \cC^-$, or ask the microstalks to be compact objects $\cD = \cC^c$. 
\end{remark}

\begin{remark}
Fix $\cD \subset \cC$ as above, and assume $\cD^{\heartsuit}$ is Artinian.  (E.g., $\cC = k-mod$ for a field $k$ and $\cD = \cC^c$.)
Then the full subcategory of $( \mu sh_{V, \C-c,o})^{\cD, \heartsuit}$ on objects with finitely stratified (rather than just locally finite)
support is also Artinian.  Indeed, any descending chain must have eventually stabilizing microsupports; we may restrict attention
to one each in the finitely many connected components of the smooth locus of the support, hence by some point, all have stabilized. 
\end{remark} 

\begin{remark}\label{remark:sheaf-vs-presheaf}
    Let $V$ be a complex contact manifold and let $L \subset V$ be a complex Legendrian. Then $ \mu sh_{V,L}(-)$ is a sheaf of stable categories while $\mathbb{P}erv_{V,L}(-)$ is a sheaf of abelian categories. Be warned however that $D(\mathbb{P}erv_{V,L}(-))$ is only a \emph{presheaf} of stable categories. 
    In particular, the natural map $D(\mathbb{P}erv_{V,L}(-)) \to \mu sh_{V,L}(-)$ may restrict to an equivalence on stalks without being an equivalence on global sections.  A very special case:  let $V= T^*S^2 \times \mathbb{\C}$ and $L= 0_{S^2} \times \{0\}$. Then $ \mu sh_{V, L}(L)= loc(S^2) \otimes loc(\mathbb{C}^*)$ while $\mathbb{P}erv_{V, L}(L)= vect_\mathbb{C} \otimes loc(\mathbb{C}^*)$, due to $S^2$ being simply connected.  Similarly, $\mu sh_{T^*\mathbb{S}^2, 0_{S^2}}(0_{S^2})= loc(S^2)$ while  $\mu sh_{T^*\mathbb{S}^2, 0_{S^2}}(0_{S^2})^{\heartsuit}= vect_\mathbb{C}$.
\end{remark}

\appendix

\section{Existence of orientation data for real symplectic manifolds \\ $\phantom{N}$ $\qquad \qquad \qquad \qquad$ by Sanath Devalapurkar} \label{sanath}

We keep the notation of Section \ref{g-od}. We write $\Sq^2: B(\mathbb{Z}/2) \to B^3(\mathbb{Z}/2)$ for the map between infinite loop spaces representing the second Steenrod operation $\Sq^2: H^*(-;\Z/2) \rightarrow H^{*+2}(-;\Z/2)$.

\begin{lemma}[{ \cite[(2.3)]{devalapurkar}}] \label{sq2 from grassmannian}
    The connecting morphism for the exact triangle
$$B^2(\Z/2) \to \tau_{\le 2}(U/O) \to \tau_{\le 1}(U/O) = B\Z \to $$ 
is the composition $B\Z \to B(\Z/2) \xrightarrow{\Sq^2} B^3(\Z/2)$.
\qed
\end{lemma}

We consider the following diagram:
\begin{equation}\label{equation:right-square}
\begin{tikzcd}[cramped]
	{U(1)} && {B\sqrt{SU}} & BU && {BU(1)=B^2\mathbb{Z}} \\
	& {B(\mathbb{Z}/2)} \\
	{B\mathbb{Z}} && {B^3(\mathbb{Z}/2)} & {\tau_{\le 3} B(U/O)} & {\tau_{\le 2} B(U/O) } & {}
	\arrow[from=1-1, to=1-3]
	\arrow[dashed, from=1-1, to=3-1]
	\arrow[from=1-3, to=1-4]
	\arrow["\alpha", dashed, from=1-3, to=3-3]
	\arrow[from=1-4, to=1-6]
	\arrow[from=1-4, to=3-4]
	\arrow["{{{\operatorname{Sq}^2}}}"', from=2-2, to=3-3]
	\arrow[from=3-1, to=2-2]
	\arrow[from=3-1, to=3-3]
	\arrow[from=3-3, to=3-4]
	\arrow[from=3-4, to=3-5]
	\arrow["{{B\operatorname{det}^2}}", "="', from=3-5, to=1-6]
\end{tikzcd}
\end{equation}
The map from $U(1) \to B\sqrt{SU}$ is the composition $U(1) \to BSU \to B\sqrt{SU}$, where the first map is the connecting map of the fiber sequence $SU \to U \xrightarrow{det} U(1)$. The bottom row is a fiber sequence, so the map $\alpha$ is induced by the fact that the composition $B\sqrt{SU} \to \pi_{\leq 2} B(U/O) \simeq B^2\mathbb{Z}$ is null. 

To establish the existence of the dotted arrow $U(1) \to B\mathbb{Z}$, it is enough to prove that the composition $B(SU) \to B(\sqrt{SU}) \to B^3(\mathbb{Z}/2)$ is null (since $U(1) \to B\sqrt{SU}$ factors through $B(SU) \to B (\sqrt{SU})$). To this end, consider the fiber sequence $SU \to U \to U(1)$. Since $U(1) = \tau_{\leq 2}U$, it follows that $SU= \tau_{\geq 3}U$. Hence $B(SU)$ is $3$-connected, so any map into $B^3(\mathbb{Z}/2)$ is null.

\begin{lemma} \label{lem: BsqrSU-factorizes}
The map $\alpha: B\sqrt{SU} \rightarrow B^3 (\Z/2)$ factors as a map of infinite loop spaces 

\[B\sqrt{SU} \xrightarrow{B\det} B(\Z/2) \xrightarrow{\Sq^2} B^3(\Z/2).\]
\end{lemma}

\begin{proof}
This follows by contemplating the diagram:
\begin{equation*}
\begin{tikzcd}
	&&& {B(\mathbb{Z}/2)} \\
	{U(1)} && {U(1)} &&&& {B(\mathbb{Z}/2)} \\
	&&& {B(\mathbb{Z}/2)} \\
	{B(SU)} && {B(\sqrt{SU})} &&&& {B^3(\mathbb{Z}/2)}
	\arrow["{=}", from=1-4, to=2-7]
	\arrow["{=}"{pos=0.7}, from=1-4, to=3-4]
	\arrow["{z \mapsto z^2}", from=2-1, to=2-3]
	\arrow[from=2-1, to=4-1]
	\arrow[from=2-3, to=1-4]
	\arrow[from=2-3, to=2-7]
	\arrow[from=2-3, to=4-3]
	\arrow["{Sq^2}", from=2-7, to=4-7]
	\arrow[from=3-4, to=4-7]
	\arrow[from=4-1, to=4-3]
	\arrow["{B\operatorname{det}}", from=4-3, to=3-4]
	\arrow["\alpha", from=4-3, to=4-7]
\end{tikzcd}
\end{equation*}
The commutativity of the leftmost square comes from the fiber sequences 
\begin{equation*}
    \begin{tikzcd}
	U && {U(1)} && BSU \\
	U && {U(1)} && {B\sqrt{SU}}
	\arrow["{\operatorname{det}}", from=1-1, to=1-3]
	\arrow["{=}"', from=1-1, to=2-1]
	\arrow[from=1-3, to=1-5]
	\arrow["{z \mapsto z^2}", from=1-3, to=2-3]
	\arrow[from=1-5, to=2-5]
	\arrow["{\operatorname{det}^2}"', from=2-1, to=2-3]
	\arrow[from=2-3, to=2-5]
\end{tikzcd}
\end{equation*}
and the rightmost square comes from \eqref{equation:right-square}. 
It remains to explain the commutativity of the bottom triangle. Since $B(SU) \to B(\sqrt{SU}) \to B(\mathbb{Z}/2)$ is a fiber sequence, it is enough to prove that the composition $B(SU) \to B(\sqrt{SU}) \to B^3(\mathbb{Z}/2)$ is null, which was proved above. 
\end{proof}

\begin{proposition}\label{prop: orientation-non-empty}    A choice of null-homotopy of $\Sq^2: B\mathbb{Z}/2 \to B^3\mathbb{Z}/2$ 
gives a section of the forgetful map from grading/orientation data 
to grading data. 
\end{proposition}
\begin{proof}
    A grading is a lift of the natural map 
    $X \to BU$ to some 
    $f: X \to B\sqrt{SU}$.  Giving orientation data is giving a null-homotopy of the map $\alpha \circ f$, which by  \Cref{lem: BsqrSU-factorizes}, can be written as 
    $\Sq^2 \circ B \det \circ f$.
\end{proof}

While $\Sq^2$ is nonzero as a map of infinite loop spaces   (as the Steenrod square is a nontrivial operation on cohomology),  it is null as a map of spaces.  Indeed, under the identifications
\begin{equation}\label{equation:bsteenrod}
\pi_0\left(\Map(B (\Z/2), B^3(\Z/2) \right) = H^3(B(\Z/2); \Z/2) = H^3(\R \P^\infty,\Z/2) = \Z/2\langle w_1^3 \rangle, 
\end{equation}
the (homotopy class of the) map $\Sq^2: B(\Z/2) \rightarrow B^3(\Z/2)$ corresponds to the element $\Sq^2(w_1) \in \Z/2\langle w_1^3\rangle$. 

But $\Sq^2(w_1) = 0$, since Steenrod squares  have the well-known property that $\Sq^n(x) = 0$ if $n > \deg(x)$. 

Now, the homotopy classes of null-homotopies of $\Sq^2$ are given by $[B(\Z/2), B^2(\Z/2)] = \mathrm{H}^2(\R \P^\infty, \Z/2) = \Z/2$.   Consider the inclusion $O \to \sqrt{SU}$; to study secondary orientation data, we will be interested in null-homotopies of the composition $\alpha': BO \to B\sqrt{SU} \xrightarrow{\alpha} B^3(\Z/2\Z)$.  The space of such null-homotopies is a torsor for $\Map(BO, B^2(\Z/2\Z))$, the homotopy classes of which are $H^2(BO, \Z/2\Z) = \Z/2\langle w_1^2, w_2\rangle$.  

One such null-homotopy arises from the null-homotopy of the factorization $O \to \sqrt{SU} \to \sqrt{SU}/O = B^2 O \xrightarrow{B w_1} B^2(\Z/2\Z)$; let us denote this null-homotopy by $\tau$.  Two more such null-homotopies arise by noticing that $\alpha'$ factors through $BO \to B\sqrt{SU} \to B(\mathbb{Z}/2) \xrightarrow{\Sq^2} B^3\mathbb{Z}/2$, and then composing with one of the two null-homotopies of $\Sq^2$. 

\begin{lemma} \label{homotopy classes of null homotopies of sq2}
    The two null-homotopies of $\alpha': BO \to B^3(\Z/2\Z)$ induced by composition with null-homotopies of $\Sq^2$ represent the classes $\tau + w_2$ and $\tau + w_2 + w_1^2$.
\end{lemma}

\begin{proof}
Denote by $\bar{\tau}$ the null-homotopy $BU \rightarrow B^3(\Z/2\Z)$ by pre-composing $\tau$ with $BU \rightarrow BO$. As the pullback map $H^2(BO,\Z/2) \rightarrow H^2(BU,\Z/2)$ is given by $w_2 \mapsto c_1$ and $w_1^2 \mapsto 0$, it is sufficient to show that the two null-homotopies both go to $\bar{\tau} + c_1$ after further composing to $BU \rightarrow BO \rightarrow B^3(\Z/2\Z)$. Thus, we consider the composition $U \rightarrow O \to \sqrt{SU} \to \sqrt{SU}/O$. The commutative diagram right above \Cref{standard structures} provides the following:
\[
    \begin{tikzcd}
	BU && BSp && B(Sp/U) && B^2 U\\
    \\
	BO && B\sqrt{SU} && B(\sqrt{SU}/O) && B^3(\Z/2\Z)
	\arrow[from=1-1, to=1-3]
	\arrow[from=1-1, to=3-1]
	\arrow[from=1-3, to=1-5]
	\arrow[from=1-5, to=1-7]
    \arrow[from=1-3, to=3-3]
	\arrow[from=1-5, to=3-5]
	\arrow[dashed, from=1-7, to=3-7]    
	\arrow[from=3-1, to=3-3]
	\arrow[from=3-3, to=3-5]
	\arrow["B^2 w_1", from=3-5, to=3-7]	\arrow[bend right=20,"\alpha", from=3-3, to=3-7]    
\end{tikzcd}
\]
We note that the dashed map exists, since any map of the form $B Sp \rightarrow B^3(\Z/2\Z)$ factorizes to $B Sp \rightarrow \tau_{\leq 3} BSp \rightarrow B^3(\Z/2\Z)$ which is canonically homotopic as $\tau_{\leq 3} BSp = 0$. In particular, the two homotopies induced from null-homotopy of $\alpha$ becomes the same after composing with $BU \rightarrow BO$ and is given by the fiber sequence $BSp \rightarrow B(Sp/U) \rightarrow B^2 U$. On the other hand, the left three horizontal maps form a map between fiber sequences. Thus, $\bar{\tau}$ comes from post-composing the fiber sequence $BU \rightarrow BSp \rightarrow B(Sp/U)$. By the proof of \Cref{spin polarization}, we see that their difference is given by $BU \xrightarrow{c_2} B^2(\Z) \rightarrow B^2(\Z/2\Z)$, which is exactly what we want.
\end{proof}

We write $\nu_+$ for the null-homotopy with class $\tau + w_2$ and $\nu_-$ for the one with class $\tau + w_2 + w_1^2$.  

\begin{lemma}
    If $\nu_{\pm}$ is used to define orientation data on a symplectic manifold $X$, then secondary orientation data on a Lagrangian $L$ is a $Pin_{\pm}$ structure on $L$. 
\end{lemma}
\begin{proof}
    Follows from Lemma \ref{homotopy classes of null homotopies of sq2} by arguing as in \Cref{lem: spin}. 
\end{proof}

\begin{lemma} \label{lem: quarternion-g/o-choice}
    Fix a stable quaternionic bundle $X \to BSp$. Then, any grading/orientation datum obtained from applying \Cref{prop: orientation-non-empty} to the canonical grading of \Cref{def: quarternionic-grading/orientation} is canonically identified with the canonical grading/orientation datum of \Cref{def: quarternionic-grading/orientation}.
\end{lemma}
\begin{proof}
    Any map $BSp \to \tau_{\le 3} B(U/O)$ canonically factors through $\tau_{\le 3} BSp = 0$; in particular, the space of such maps is contractible.  We defined the canonical grading/orientation data by taking the corresponding null-homotopy of the composition $X \to BSp \to BU \to \tau_{\le 3} B(U/O)$.  

    Meanwhile the grading/orientation data from \Cref{prop: orientation-non-empty} are induced by choices of null-homotopy of the map $\alpha: B\sqrt{SU} \rightarrow B^3(\Z/2)$. Any such null-homotopy induces a null-homotopy out of $B Sp$ by the pre-composition
    \[B Sp \rightarrow B\sqrt{SU} \rightarrow B^3(\Z/2)\] as explained in \Cref{fiber versus canonical grading}. 
\end{proof}

\section{$t$-structures on Fukaya categories} \label{app: Fukaya}

Here we translate our main results across the sheaf/Fukaya correspondence of \cite{GPS3}, in order to construct $t$-structures on Fukaya categories of certain complex exact symplectic manifolds with contracting weight $1$ $\mathbb{C}^*$ action, such as conic symplectic resolutions and moduli of Higgs bundles. 

Recall that a
\emph{Liouville manifold} is a (real) exact symplectic manifold $(W, \lambda)$ which is modeled at infinity on the symplectization of a contact manifold. The negative flow of the Liouville vector field $Z$ (defined by $\lambda = d\lambda(Z, \cdot)$) retracts $W$ onto a compact subset $\mathfrak{c}_W$ called the \emph{core}. Fix if desired some larger closed conic subset $\Lambda \supset \c_W$.  We say an exact Lagrangian $L \subset W$ is {\em admissible} if it is closed and, outside a compact set, 
it is conic and disjoint from $\Lambda$.  For example, when $Z$ is Smale and gradientlike for a Morse function, 
the ascending trajectories from maximal index critical points (``cocores'') are admissible.   
Any admissible Lagrangian disk which meets $\Lambda$ transversely at a single smooth point is termed
a generalized cocore; we say a set $\{\Delta_\alpha\}$ of generalized cocores is complete if it meets
every connected component of the smooth locus of $\Lambda$.  If $Z$ is gradient-like for a Morse-Bott function, then
$\Lambda$ is known to admit a complete set of generalized cocores.

Fix grading and orientation data.  We recall that one source of
such data is a polarization of the stable 
symplectic normal (or equivalently tangent) bundle 
as explained in \Cref{g-od} above; see also \cite[Sec. 5.3]{GPS3}.
Then one can define a partially wrapped Fukaya category 
$Fuk(W, \partial_\infty \Lambda)$ \cite{GPS2}. 
Objects 
are provided admissible Lagrangians equipped with structures corresponding to the grading and orientation data.  
The object associated to a generalized cocore is unique up to grading shift. 
The completion of $Fuk(W, \partial_\infty \Lambda)$ with respect to exact triangles 
is generated by any complete collection 
of generalized cocores \cite{CDGG, GPS2}.  We further complete with respect to idempotents, and still
denote the resulting category $Fuk(W, \partial_\infty \Lambda)$.

\begin{definition} \label{fukaya t} 
Fix a collection $\{\Delta_\alpha\}$ of generalized cocores equipped with grading data.  We define: 
$$Fuk(W, \partial \Lambda)^{\ge 0} = \{L \, | \, \Hom(L, \Delta_\alpha) \mbox{ is concentrated in degrees $\ge 0$}\}$$
$$Fuk(W, \partial \Lambda)^{\le 0} = \{L \, | \, \Hom(L, \Delta_\alpha) \mbox{ is concentrated in degrees $\le 0$}\}$$
\end{definition}

\noindent It is natural to ask when \Cref{fukaya t} determines a $t$-structure. 

\vspace{2mm} 

Suppose now that $W$ is a complex manifold and that $d\lambda = \operatorname{Re}\omega_\mathbb{C}$, for some complex symplectic structure $\omega_{\mathbb{C}}$ on $W$. 
As we have remarked above, and is explained in detail in \Cref{g-od} above, such a $W$ 
carries canonical grading data, which agrees with the grading induced by any stable complex Lagrangian polarization
of the stable complex symplectic normal bundle (viewed as a real polarization of the real stable symplectic normal bundle), 
and any complex Lagrangian carries a canonical secondary grading.  
More generally, consider a real Lagrangian (or union of Lagrangians) $L \subset W$.  Then 
$TL$ determines a section of $LGr(TW)|_L$.  By a $g$-complex structure on $L$, we mean a simply
connected neighborhood of $LGr_{\C}(TW)|_L \subset LGr(TW)|_L$ containing $TL$. 
An $g$-complex Lagrangian has a canonical secondary grading.  
Evidently any Lagrangian disk transverse to a complex Lagrangian admits a canonical $g$-complex structure.   

\vspace{2mm}

We recall the main result of \cite{GPS3}.  
Assume $W$ is real analytic and Liouville, and $\Lambda \supset \c_W$ is subanalytic, Lagrangian at smooth points,
and admits a complete collection of generalized cocores $\{\Delta_\alpha\}$.  
Fix grading and orientation data coming from a stable polarization.\footnote{The dependence on polarizations here and 
henceforth could be removed by a version of \cite[Sec. 11]{nadler-shende} for Fukaya categories.}
Then $Fuk(W, \partial_\infty \Lambda) \cong \mu sh_\Lambda(\Lambda)^{c, op}$ 
carrying  $\Delta_\alpha$ to co-representatives of microstalk functors (here $(-)^c$ means that we take compact objects). 
Chasing definitions reveals that under \cite{GPS3}, the normalized microstalks of
\Cref{corollary:microstalk-intro-test} are carried to the canonically graded $\Delta_\alpha$ (up to some universal shift).  
Therefore: 

\begin{corollary} \label{corollary:t-structure-fukaya}
Suppose $(W, \lambda)$ is a complex exact symplectic manifold with weight $1$ $\mathbb{C}^*$-action and 
$\Lambda \subset W$ a conic complex (singular) Lagrangian.  Assume 
$(W, \re(\lambda))$ is Liouville and $\Lambda$ admits a complete collection
of generalized cocores $\{\Delta_\alpha\}$.  

Fix grading and orientation data induced by a stable complex Lagrangian polarization of 
$(W, d\lambda)$.  Then the equivalence of \cite{GPS3} carries the $t$-structure 
of \Cref{theorem:main symplectic intro} to a shift of Definition \ref{fukaya t}, which therefore provides
a $t$-structure. 
\end{corollary}
Also by \Cref{theorem:main symplectic intro} translated through \cite{GPS3}, 
if $L, M \subset \c_X$ are 
 spin compact (necessarily conic) smooth Lagrangians,  
their Floer cohomology matches the cohomology of a (shifted) perverse sheaf supported on $L \cap M$.

The hypotheses of \Cref{corollary:t-structure-fukaya} are obviously satisfied for $W$ a cotangent bundle of a complex manifold.  
More generally, there are many examples of holomorphic symplectic manifolds with a weight $1$ $\C^*$ action
scaling the symplectic form -- (coloop-free) quiver varieties, moduli of Higgs bundles, etc. --
which satisfy all the hypotheses (the Liouville flow is gradientlike for 
the moment map for $S^1 \subset \C^*$, which is Morse); cf.\ \cite{vzivanovic2022exact}. 

\vspace{2mm}

The Fukaya category has the advantage that non-conic Lagrangians directly define objects, 
to which we thus have more direct geometric access.  In particular, we can now construct objects in the heart. 
The point is that for index reasons, if $L \cup M$ is $g$-complex, then the Floer homology 
between $L$ and $M$ must be concentrated in degree zero.  
Similar considerations, for cotangent bundles, appear in 
\cite{jin-perverse}. 
Thus: 

\begin{corollary} \label{exact objects}
Retain the hypotheses of Corollary \ref{corollary:t-structure-fukaya}. 
Let  $L \subset W$ be an exact Lagrangian.  Assume $L$ is compact, or more generally, that 
$\partial_\infty L$ wraps into $\partial_\infty \Lambda$ without passing through any 
$\partial_\infty \Delta_\alpha$.
Suppose  $L \cup \Delta_\alpha$ is $g$-complex 
for each $\alpha$.  
Then $L \in Fuk(W, \partial_\infty \Lambda)^\heartsuit = \mu sh_\Lambda(\Lambda)^\heartsuit$.
\end{corollary} 

The wrapping hypotheses appears to ensure that Hom in the wrapped Fukaya category is in fact computed
without any wrapping, i.e. just by the Floer homology.  One can imagine 
applying the 
corollary by taking an exact holomorphic Lagrangian is asymptotic to $\partial_\infty \Lambda$, 
and cutting off and straightening; such a process would plausibly produce a Lagrangian
satisfying the wrapping hypothesis.  We contemplate this process because such asymptotically conical exact 
holomorphic Lagrangians appear frequently in examples, but do not literally provide objects of 
$Fuk(X, \partial_\infty \Lambda)$.  

\vspace{2mm}

In a different direction, it was observed in \cite{shende-fibers} that the equivalence of microsheaf and Fukaya categories  \cite{GPS3}
remains true after enlarging the Fukaya category to contain
unobstructed compact nonexact Lagrangians, of course taking coefficients in the Novikov field.   As there, 
this observation is profitably combined with the fact  \cite{solomon-verbitsky} that holomorphic Lagrangians are unobstructed in 
hyperk\"ahler manifolds.  We conclude: 

\begin{corollary} \label{nonexact objects}
Retain the hypotheses of Corollary \ref{corollary:t-structure-fukaya}. 
Assume $W$ is hyperk\"ahler and let  $L \subset W$ be a compact holomorphic Lagrangian.  
Suppose 
 $L \cup \Delta_\alpha$ is $g$-complex for each $\alpha$.  
Then $L \in Fuk(W, \partial_\infty \Lambda)^\heartsuit = \mu sh_\Lambda(\Lambda)^\heartsuit$,
all categories taken with coefficients over the Novikov field. 
\end{corollary} 

 \Cref{exact objects} and \Cref{nonexact objects} can be stated more generally as 
defining fully faithful functors from appropriate abelian category of local systems on $L$ to the heart of the $t$-structure.

\bibliographystyle{plain}
\bibliography{refs}
 
\end{document}